\documentclass[10pt, oneside, a4paper]{article}
\usepackage{geometry}
\title{Bounds on special values of $L$-functions of \\ elliptic curves  in an Artin-Schreier family}
\author{ {\large Richard Griffon}  \\
{\normalsize\it Mathematisch Instituut, 
Universiteit Leiden} } 
\date{} 

\usepackage[english]{babel}
\usepackage[T1]{fontenc}
\usepackage[utf8]{inputenc} 
\usepackage{lmodern}
\usepackage{amssymb,amsmath,amscd}  
\usepackage{mathrsfs}
\usepackage{ntheorem}
\usepackage{multicol}

			{\theoremstyle{plain}
			\theoremseparator{ --}
			\newtheorem{theo}{Theorem}[section]}

			{\theoremstyle{plain}
			\theoremseparator{ --}
			\newtheorem{coro}[theo]{Corollary}}

			{\theoremseparator{ --}
			\newtheorem{lemm}[theo]{Lemma}}

			{\theoremseparator{ --}
			\newtheorem{prop}[theo]{Proposition}}

			{\theoremseparator{ --}
			}

			{\theoremseparator{ --}
			\newtheorem{itheo}{Theorem}}

			{\theorembodyfont{\normalfont}
			\newtheorem{defi}[theo]{Definition}}

			{\theoremstyle{plain}
			\theorembodyfont{\normalfont}
			}

			{\theoremstyle{plain}
			\theorembodyfont{\normalfont}
			\newtheorem{rema}[theo]{Remark}}

			{\theoremheaderfont{\itshape}
			\theorembodyfont{\normalfont}
			\theoremseparator{: }
			\newtheorem*{proof}{Proof}}

			\newcommand{\ProofEnd}{\hfill $\Box$}

\usepackage{tabularx}
\usepackage{array}
\usepackage{multirow}
\usepackage{multicol}
\usepackage{enumerate}
\usepackage{graphicx}
\usepackage[table, dvipsnames]{xcolor}	
 
\usepackage{fancyhdr}

\usepackage{url}

			\DeclareMathOperator{\rk}{rank}
			\DeclareMathOperator{\Reg}{Reg}
			\DeclareMathOperator{\ord}{ord}
			\DeclareMathOperator{\Gal}{Gal}
			\DeclareMathOperator{\DEGRE}{deg }
			\renewcommand{\deg}{\DEGRE}

			\newcommand{\R}{\ensuremath{\mathbb{R}}}
			\newcommand{\Q}{\ensuremath{\mathbb{Q}}}
			\newcommand{\C}{\ensuremath{\mathbb{C}}}
			\newcommand{\Z}{\ensuremath{\mathbb{Z}}}
			\newcommand{\G}{\ensuremath{\mathbb{G}}}
			\renewcommand{\P}{\ensuremath{\mathbb{P}}}
			\newcommand{\F}{\ensuremath{\mathbb{F}}}
			\renewcommand{\H}{\ensuremath{\mathrm{H}}}

			\newcommand{\Qbar}{\ensuremath{\bar{\mathbb{Q}}}}
			\newcommand{\Qellbar}{\bar{\Q_\ell}}  

			\newcommand{\Ecal}{\mathcal{E}}
			
			\newcommand{\Ccal}{\mathcal{C}}

			\newcommand{\e}{\ensuremath{\mathrm{e}}}
			\newcommand{\dd}{\ensuremath{\,\mathrm{d}}}
			\newcommand{\into}{\hookrightarrow}
			
			\renewcommand{\bar}[1]{\ensuremath{\overline{#1}}}
			\newcommand{\hhat}[1]{\ensuremath{\widehat{#1}}}

			\newcommand{\tam}{\tau}
			\newcommand{\cond}{\mathcal{N}}
			\newcommand{\type}[1]{\mathbf{#1}}
			\newcommand{\tors}{_{\mathrm{tors}}}
			\newcommand{\sep}{^\mathrm{sep}}

			\newcommand{\ie}{\textit{i.e.}{}}
			\newcommand{\cf}{\textit{cf.}{}}

			\input cyracc.def
			\DeclareFontFamily{U}{russian}{}
			\DeclareFontShape{U}{russian}{m}{n}
				{ <5><6> wncyr5
				<7><8><9> wncyr7
				<10><10.95><12><14.4><17.28><20.74><24.88> wncyr10 }{}
			\DeclareSymbolFont{Russian}{U}{russian}{m}{n}
			\DeclareSymbolFontAlphabet{\mathcyr}{Russian}
			\makeatletter
			\let\@math@cyr\mathcyr
			\renewcommand{\mathcyr}[1]{\@math@cyr{\cyracc #1}}
			\makeatother
			\newcommand{\sha}{\mathcyr{SH}}
			\newcommand{\ka}{\mathcyr{K}}
			 
			\usepackage{bm}
			\newcommand{\trivcar}{ \bm{1}}

			\usepackage{bbm}
			\newcommand{\oomega}{\hm{\omega}}

			\newcommand{\shG}{\mathcal{G}}
			\newcommand{\shKl}{\mathcal{K}l}
			\newcommand{\norm}{\mathbf{N}}
			\DeclareMathOperator{\trace}{tr}
			\newcommand{\Trace}{\mathrm{Trace}\,} 

			\newcommand{\mom}{\mathscr{M}}
			\newcommand{\dis}{\mathcal{D}^\ast}
			\newcommand{\totvar}[1]{\mathscr{V}(#1)}
			\newcommand{\frob}{\mathrm{Frob}}
			\newcommand{\Frob}{\mathrm{Fr}}

			\newcommand{\Kl}{\mathrm{Kl}}
			\newcommand{\Kln}{\ka} 
			\newcommand{\kkappa}{\ensuremath{{\mathsf{kl}}}}
			\newcommand{\kkappan}{\ensuremath{\hm{\mathsf{kl}}}}
			\newcommand{\ang}{\theta}
			\newcommand{\angn}{\bm{\theta}}
			\newcommand{\muST}{\mu_{{\rm ST}}} 

			\newcommand{\symm}{\mathrm{Symm}}
			\newcommand{\GL}{\mathrm{GL}}
			\newcommand{\SL}{\mathrm{SL}}
			\newcommand{\SU}{\mathrm{SU}}

\usepackage[colorlinks=true,
citecolor= Orange, 
linkcolor= Mahogany, 
urlcolor=black!60, 
pagebackref
]{hyperref}

			\renewcommand{\ss}{^\mathrm{s.s.}}

\begin{document}
\pagestyle{fancy}

\maketitle 

\noindent\hfill\rule{7cm}{0.5pt}\hfill\phantom{.}

\paragraph{Abstract --} 
We study a certain Artin-Schreier family of elliptic curves over the function field $K=\F_q(t)$. 
We prove an asymptotic estimate on the special values of their $L$-function in terms of the degree of their conductor; 
we show that the special values are, in a sense, ``asymptotically as large as possible''. 
In the paper we also provide an explicit expression for their $L$-function.

The proof of the main result uses this expression and a detailed study of the distribution of character sums related to Kloosterman sums.
Via the BSD conjecture, the main result translates into an analogue of the Brauer--Siegel theorem for these elliptic curves.

\medskip

 
\smallskip
\noindent{\it 2010 Math. Subj. Classification:}  
11G05, 
11G40, 
11M38, 
11F67, 
11L05, 
11J20. 
 
\smallskip
\noindent{\it E-mail:} \href{mailto:r.m.m.griffon@math.leidenuniv.nl}{{r.m.m.griffon@math.leidenuniv.nl}}
\medskip

\noindent\hfill\rule{7cm}{0.5pt}\hfill\phantom{.}

\section*{Introduction}
\pdfbookmark[0]{Introduction}{Introduction} 
\addcontentsline{toc}{section}{Introduction}

\setcounter{section}{0}

\paragraph{}
Let $\F_q$ be a finite field of odd characteristic $p$ and $K=\F_q(t)$.
Consider a nonisotrivial elliptic curve~$E$ defined over $K$, and  its associated $L$-function%
\footnote{Since the base field $K$ is fixed  and all the invariants of $E$ we consider are relative to $K$, we drop the dependency on $K$ from the notations.} $L(E, T)$.
\emph{Via} a cohomological interpretation, Grothendieck has proved that, 
even though $L(E, T)$ is \emph{a priori} defined as a formal power series in $T$, it is actually a polynomial with integral coefficients, whose degree we denote by $b(E)$.
Moreover, $L(E, T)$ satisfies the expected functional equation relating $L(E, T)$ to $L(E, 1/q^2T)$. 

Define $\rho(E)$ to be the order of vanishing of $L(E, T)$ at the central point $T=q^{-1}$ 
 and the \emph{special value} of $L(E, T)$ by $L^\ast(E, 1):= \lim_{T\to q^{-1}} (1-qT)^{-\rho(E)} \cdot L(E, T)$.
These invariants both appear in the conjecture of Birch and Swinnerton-Dyer\footnote{Hereafter abbreviated as BSD} through which they are related to ``arithmetic'' invariants of $E$. 


We will be interested in comparing the size of the special value 
$L^\ast(E, 1)$ to the degree $b(E)$ of the $L$-function.
It is relatively straightforward to prove  that 
\begin{equation}\label{ieq.trivbounds}
-1 \leq \frac{\log L^\ast(E, 1)}{\log\left(q^{b(E)}\right)} \leq 0 + o(1) \qquad (\text{as } b(E)\to\infty),
\end{equation}
and it seems natural to ask about the optimality of such bounds.
 In other words, given a family $\mathscr{E}$ of nonisotrivial elliptic curves $E$ over $K$ 
 with $b(E)\to\infty$, we investigate the asymptotic behaviour of the ratio 
${\log(L^\ast(E, 1))}/{\log\left(q^{b(E)}\right)}$ as $E$ runs through $\mathscr{E}$. 
Does this ratio have a limit? If so, what is this limit?

These questions are still wide open and, as far as the author knows, they have only been settled for six special families $\mathscr{E}$ (see \cite[Thm.\,7.12]{HP15}, \cite[Coro.\,5.1]{Griffon_Legendre}, \cite[Thm.\,4.2]{Griffon_Hessian} and \cite[Thm.\,9]{Griffon_PHD}).  
These six examples are known as ``Kummer families'' of elliptic curves: one obtains them by pulling-back an elliptic curve $E_1/K$ by the map $t\mapsto t^d$ for larger and larger integers $d$ which are coprime to $q$. 

In those cases, the ratio in \eqref{ieq.trivbounds} does have a limit, and this limit is $0$.


\paragraph{}
In this article, we answer the two questions above for an ``Artin--Schreier family'' of elliptic curves over~$K$. 
More precisely, we prove 

\begin{itheo}\label{itheo.main}
Let $\F_q$ be a finite field of odd characteristic $p$  and $K=\F_q(t)$. 
Fix $\gamma\in\F_q^\times$ and, for all  integers $a\geq 1$, let $E_{a,\gamma}$ be the elliptic curve over $K$ given by the affine Weierstrass model
\begin{equation}\notag{} 
E_{a,\gamma}:\quad y^2=x\left(x+16\gamma \right)\left(x+\wp_a(t)^2\right), 
\quad \text{ with } \wp_a(t) =t^{q^a} -t\in\F_q[t]. 
\end{equation}
Then, as $a\to\infty$, the limit below exists and is $0$: 
\begin{equation}\label{ieq.main}
\lim_{a\to\infty} 
\frac{\log L^\ast(E_{a, \gamma}, 1)}{\log\left(q^{b(E_{a, \gamma})}\right)} = 0.
\end{equation}
\end{itheo}

The name of ``Artin--Schreier family'' stems from its construction: starting with the elliptic curve $E_{\gamma}/K$ given by ${y^2=x(x+16\gamma)(x+t^2)}$, one obtains $E_{a, \gamma}$ by pulling back $E_{\gamma}$ by the Artin--Schreier map $t\mapsto \wp_a(t)$ for any $a\geq1$.
This family of elliptic curves $\{E_{a, \gamma}\}_{a\geq 1}$  was previously studied in 
\cite{UlmerPries} where, among others, the authors prove that $E_{a, \gamma}$  satisfies the BSD conjecture.
Following \cite[\S7.3]{UlmerPries}, we note the ressemblance between $E_{a, \gamma}$ and a Legendre elliptic curve. 
We refer to Theorem \ref{theo.bnd.spval} for a more quantitative version of \eqref{ieq.main}.

Once again, prior to Theorem \ref{itheo.main}, the only known examples of families of elliptic curves 
exhibiting a behaviour such as \eqref{ieq.main}
were Kummer families.

\paragraph{} 
From Theorem \ref{itheo.main} and from the BSD conjecture, we will deduce (see Theorem \ref{theo.BS}): 
\begin{itheo}\label{itheo.BS}
Let $\F_q$ be a finite field of odd characteristic and $K=\F_q(t)$. 
For any $\gamma\in\F_q^\times$ and any integer~$a\geq 1$, the Tate--Shafarevich group $\sha(E_{a, \gamma})$ is finite.
Furthermore, one has
\[\lim_{a\to\infty} \frac{ \log\big(|\sha(E_{a, \gamma})|\cdot\Reg(E_{a, \gamma}) \big)}{\log H(E_{a, \gamma})} = 1.\]
\end{itheo}

Following \cite{HP15}, we view this result as an analogue of the Brauer--Siegel theorem for the elliptic curves~$\{E_{a, \gamma}\}_{a\geq 1}$.
We further comment on this result in section \ref{sec.BS}. 

\paragraph{}
Let us give an outline of the proof of Theorem \ref{itheo.main} as we describe the structure of the paper and state the other results contained in it.
In section \ref{sec.family}, we start by reviewing the construction of the elliptic curves $E_{a, \gamma}$ and by computing some of their invariants. 
We also recall the definition of their $L$-function and state the BSD conjecture (proved by Pries and Ulmer for $E_{a, \gamma}$, \cf{} \cite{UlmerPries}).

In the two following sections, we compute the $L$-function of $E_{a, \gamma}$; 
the relevant objects are introduced in section \ref{sec.prelim}. 
In particular, a central role is played by angles of some Kloosterman sums. 
For the purpose of this introduction  let us only say that, to any place $v\neq 0, \infty$ of $K$, we will attach a character sum $\Kln_\gamma(v)$.  
The sum $\Kln_\gamma(v)$ is a real number satisfying $|\Kln_\gamma(v)|< 2q^{d_v/2}$ where $d_v$ is the degree of $v$. 
Hence there exists an angle $\theta_v\in (0, \pi)$ such that $\Kln_\gamma(v) = 2q^{d_v/2}\cdot\cos\theta_v$. 
The reader is referred to \S\ref{subsec.kloos.norm} and \S\ref{subsec.anglesnorm} for precise definitions.
Section \ref{sec.Lfunc} is devoted to the calculation 
of the $L$-function itself, which 
results in the following expression: 

\begin{itheo}\label{itheo.Lfunc} For all integers $a\geq 1$, we denote by $P_q(a)$ the set of places $v\neq 0, \infty$ of $K$ whose degree $d_v$ divides $a$. 
Then, for all $\gamma\in\F_q^\times$, the $L$-function of $E_{a, \gamma}$ is given by
\begin{equation}\label{ieq.Lfunc}
L(E_{a, \gamma}, T)
= \prod_{v\in P_q(a)} \left(1-(qT)^{d_v}\right)
\left(1-\e^{2i\theta_v}(qT)^{d_v}\right)
\left(1-\e^{-2i\theta_v}(qT)^{d_v}\right),
\end{equation}
where,  for all $v\in P_q(a)$,  $\theta_v\in(0,\pi)$ is as above (see \S\ref{subsec.kloos.norm} and \S\ref{subsec.anglesnorm} for a precise definition).
\end{itheo}
This result is proved by a ``point-counting'' argument, directly from the definition of $L(E_{a, \gamma}, T)$, through manipulations of character sums over finite fields. 
Given the paucity of tables of $L$-functions of elliptic curves over $K$ of large conductor, such an explicit expression of $L(E_{a, \gamma}, T)$ can be of independent interest.

As a by-product, Theorem~\ref{itheo.Lfunc} yields a closed formula for the analytic rank $\ord_{T=q^{-1}}L(E_{a, \gamma}, T)$. 
Using that the BSD conjecture is proven for $E_{a, \gamma}$, we recover a result of \cite{UlmerPries} stating that the ranks of $E_{a, \gamma}(K)$ are unbounded as $a\to\infty$; 
more precisely, we show in Corollary \ref{coro.unbnd.rk} that:
$\rk E_{a, \gamma}(K) \gg_q {q^a}\big/{a}$. 

From Theorem \ref{itheo.Lfunc}, we also derive (Proposition \ref{prop.anal.rk} and \eqref{eq.expr.spval.alt}) an explicit expression for the special value $L^\ast(E_{a, \gamma}, 1)$. 
In the notations of \eqref{ieq.Lfunc}, we obtain an expression of the shape
\begin{equation}\label{ieq.est.spval}
\frac{\log L^\ast(E_{a, \gamma}, 1)}{b(E_{a, \gamma})} 
= \frac{\log\left(\substack{\text{positive} \\ \text{integer}}\right)}{b(E_{a, \gamma})}   + \frac{1}{b(E_{a, \gamma})} \cdot\sum_{v\in P_q(a)} \log \sin^2\theta_v.
\end{equation} 
Estimating the first term is straightforward: it is $o(1)$ as $a\to\infty$ and thus does not play any role in the asymptotics. 
In order to prove the limit in Theorem \ref{itheo.main}, we have  to investigate the behaviour of the second term: specifically, we need to show that it is $o(1)$ too. 
The size of this term visibly depends on how the angles $\{\theta_v\}_{v\in P_q(a)}$ distribute in the interval $[0,\pi]$.
Since $t\mapsto \log \sin^2 t$ tends to $-\infty$ 
at $0$ and $\pi$, this size also depends on how close the angles 
can  be to the endpoints of $(0,\pi)$.

Therefore we devote sections \ref{sec.small.angles} and \ref{sec.distrib} to studying the distribution of the angles $\{\theta_v\}_{v\in P_q(a)}$  in more detail. 
The two main results in these sections can be stated as follows: 

\begin{itheo}\label{itheo.angles}
 Notations being as above,
\begin{enumerate}[(i]
\item\label{itheo.ST.eff}- Theorem \ref{theo.ST.eff})
For all continuously differentiable functions $g:[0, \pi]\to\C$, 
\begin{equation}\notag{}
\left|\frac{1}{|P_q(a)|}\sum_{v\in P_q(a)} g(\theta_v) 
- \frac{2}{\pi}\int_0^\pi g(t) \cdot \sin^2 t\dd t\right|
\ll_q \frac{a^{1/2}}{q^{a/4}}\cdot \int_0^\pi |g'(t)|\dd t.
\end{equation}
\item\label{itheo.minangle}- Corollary \ref{coro.minangle})
There is a constant $c>0$ such that $\min\{\theta_v, \pi-\theta_v\} \geq (q^a)^{-c}$ for all $v\in P_q(a)$. 
\end{enumerate}
\end{itheo}

By the work of Katz, it is known that the angles of Kloosterman sums become equidistributed  in $[0, \pi]$ with respect to the Sato--Tate measure (see \cite{Katz_GKM}). 
It turns out that the same statement holds for the angles $\{\theta_v\}_{v\in P_q(a)}$ (see Theorem \ref{theo.ST}). 
The proof relies on an adaptation of Katz's method in \cite[Chap. 3]{Katz_GKM} and results of Fu and Liu in \cite{FuLiu}.
This equidistribution result, however, is not sufficient for our purpose: we need a more effective version such as Theorem \ref{itheo.angles}\eqref{itheo.ST.eff}. 
The effective version (Theorem \ref{theo.ST.eff}) will follow from Theorem \ref{theo.ST} after a more detailed analysis using tools from equidistribution theory (see \cite{Niederreiter_Klo}).

The main goal of section \ref{sec.small.angles} is to prove Theorem \ref{itheo.angles}\eqref{itheo.minangle}; 
we actually prove a more general result there  
(see Theorem \ref{theo.minangle}). 
The proof has a diophantine approximation flavour and the main tool 
is a version of Liouville's inequality (as in \cite{MiWa}). 

\paragraph{}
Combining the two results in Theorem \ref{itheo.angles} and approximating $t\mapsto\log\sin^2 t$ by sufficiently regular functions,
we prove (Theorem \ref{theo.ST.ext}) that 
\begin{equation}\label{ieq.ST.ext}
\frac{1}{|P_q(a)|} \sum_{v\in P_q(a)} \log\sin^2\theta_v \xrightarrow[a\to\infty]{}
\frac{2}{\pi}\int_0^\pi \log\sin^2 t \cdot \sin^2 t\dd t  = \log(\e/4).
\end{equation}

Theorem \ref{itheo.main} then follows rather easily (see Theorem \ref{theo.bnd.spval}), 
since \eqref{ieq.ST.ext} implies that 
the second term in \eqref{ieq.est.spval} is indeed $o(1)$ as $a\to\infty$.  
Finally, section \ref{sec.BS} is devoted to the proof of Theorem \ref{itheo.BS} (see Theorem \ref{theo.BS}).

\numberwithin{equation}{section}    
\section[The Artin-Schreier family of elliptic curves under consideration]{The Artin-Schreier family of elliptic curves $E_{a, \gamma}$} 
\label{sec.family}

\begin{center}
\emph{Throughout this article, we fix a finite field $\F_q$ of characteristic $p\geq 3$, \\
and we denote by $K=\F_q(t)$ the function field of the projective line $\P^1/\F_q$.}
\end{center}

In this section, we explain in some detail how the curves $E_{a, \gamma}$ are constructed and we collect elementary facts about them. 
We also setup some notations and conventions that will be in force for the rest of the paper. 
For a nice account of the theory of elliptic curves over $K$, the reader can confer \cite{UlmerParkCity}. 
 
For all integers $a\geq 1$, we let $\wp_a(t) = t^{q^a}-t\in\F_q[t]$. 
For any $\gamma\in\F_q^\times$ and any $a\geq 1$, we consider the elliptic curve $E_{a, \gamma}$ defined over $K$ by the affine Weierstrass model:
\begin{equation}\label{eq.Wmodel}
E_{a,\gamma}:\quad y^2=x\left(x+16\gamma \right)\left(x+\wp_a(t)^2\right). \end{equation}

The sequence $\{E_{a, \gamma}\}_{a\geq 1}$ is called an \emph{Artin-Schreier family} of elliptic curves over $K$. 
This terminology comes from the following observations. 
Let $E_\gamma/K$ be the elliptic curve given by $y^2=x(x+16\gamma)(x+t^2)$. Then, for all $a\geq 1$, the curve $E_{a, \gamma}$ is the pullback of $E_\gamma$ under the Artin-Schreier map $t\mapsto \wp_a(t)$ of $\P^1_{/\F_q}$. 
Hence, studying $E_{a, \gamma}$ over $K=\F_q(t)$ is equivalent to studying $E_\gamma$ over the Artin-Schreier extension $\F_q(u_a)$ of $\F_q(t)$, where $\wp_a(u_a)=t$.

We remark that $E_{a, \gamma}$ is ``almost'' a Legendre curve. 
More precisely, in the setting of \cite{BaigHall}, the curve~$E_{a, \gamma}$ can be obtained as follows. 
Starting from the Legendre elliptic curve $E_{0, \gamma}/K$ given by
\[E_{0, \gamma} : \quad y^2 = x(x-1)\left(1-(16\gamma)^{-1}\cdot t^2\right),\]
one takes a quadratic twist by $-(16\gamma)^{-1}$, obtaining $E'_{0, \gamma}$ defined by $ y^2 = x(x+1)\left(1+(16\gamma)^{-1}\cdot t^2\right)$. 
Pulling back $E'_{0, \gamma} $ along $\wp_a:\P^1\to\P^1$, one recovers the curve $E_{a, \gamma}$ defined by \eqref{eq.Wmodel}.

\paragraph{}
These elliptic curves $E_{a, \gamma}$ were studied in \cite[\S6.4, \S7.3]{UlmerPries} where, among others, it was shown that they satisfy the BSD conjecture (see \S\ref{subsec.BSD}). 
In their paper, Pries and Ulmer describe an alternative construction of the curves $E_{a, \gamma}$, which we now recall for later use. 
For any $\gamma\in\F_q^\times$, we let $f_\gamma:\P^1_{/K}\to\P^1_{/K}$ be the map $[x_0:x_1]\mapsto x_0/x_1+\gamma x_1/x_0$. 
Consider the curve $Z_{a, \gamma}\subset\P^1_{/K}\times\P^1_{/K}$ defined over $K$ by 
\[ f_\gamma([x_0:x_1])-f_\gamma([y_0:y_1]) =\wp_a(t),\]
in the  $([x_0: x_1], [y_0:y_1])$-coordinates  on $\P^1_{/K}\times\P^1_{/K}$. 
This curve is smooth of  genus $1$ and admits one $K$-rational point $([0:1], [0:1])$. 
The curve $Z_{a, \gamma}$ is given in the affine $(x,y)$-coordinates  on $\mathbb{A}^2_{/K}$ by 
 \[ (x-y)(xy-\gamma)=\wp_a(t)\cdot xy.\]
 The change of coordinates
 $\displaystyle (x, y)\mapsto (u,v)=\left( -\gamma\cdot\frac{\wp_a(t) -x+y}{x-y}, \gamma\wp_a(t)\cdot\frac{y(\wp_a(t)+y-x) +2\gamma}{x^2-y^2} \right)$
 then brings $Z_{a, \gamma}$ into the affine Weierstrass form
\begin{equation}\label{eq.Wmod.X}
E^\circ_{a, \gamma}:\qquad v^2 -\wp_a(t)\cdot uv = u^3-2\gamma\cdot u^2 +\gamma^2\cdot u. 
\end{equation}
One finally passes from $E^\circ_{a, \gamma}$ to $E_{a, \gamma}$ by means of the $2$-isogeny  $\phi:E^\circ_{a, \gamma}\to E_{a, \gamma}$ given by
\[ (u, v)\mapsto (x, y)=\left( \frac{4v(v-\wp_a(t)u)}{u^2}, \frac{4(2v-\wp_a(t)u)(u^2-\gamma^2)}{u^2}\right).\]
  
From the model \eqref{eq.Wmodel}, it is straightforward to compute the $j$-invariant $j(E_{a, \gamma})$ of $E_{a, \gamma}$ and obtain that
\[j(E_{a, \gamma}) = \frac{\left(\wp_a(t)^4 - 16\gamma\cdot \wp_a(t)^2 + 2^8 \gamma^2\right)^3}{\gamma^2 \cdot \wp_a(t)^4\cdot\left(\wp_a(t)^2 - 16\gamma\right)^2}\in K.\]
As a rational function of $t$, the $j$-invariant $j(E_{a, \gamma})$ is visibly nonconstant and separable. 
In particular, the curve $E_{a, \gamma}$ is not isotrivial.

\subsection{Bad reduction and invariants}
\label{subsec.invariants}

For any place $v$ of $K$, we denote by $d_v$ or $\deg v$ the degree of $v$ and $\F_{v}$ the residue field at $v$ (an extension of $\F_q$ of degree $d_v$). 
We identify finite places of $K$ with monic irreducible polynomials in $\F_q[t]$;
we also identify the residue field at $v\neq\infty$ with $\F_q[t]/(B_v)$ if $B_v\in\F_q[t]$ is the monic irreducible polynomial corresponding to $v$. 

Let us describe the reductions of $E_{a, \gamma}$ at places of $K$ and compute its relevant invariants. 
A straightforward computation of the discriminant of the model \eqref{eq.Wmodel} of $E_{a, \gamma}$ gives that
\begin{equation}\label{eq.disc}
\Delta = 2^{12} \gamma^2 \cdot \wp_a(t)^4\cdot\left(\wp_a(t)^2 - 16\gamma\right)^2.
\end{equation}
The finite places of bad reduction of $E_{a, \gamma}$ are then the monic irreducible divisors of $\Delta$ in $\F_q[t]$. 
By a routine application of Tate's algorithm (see \cite[Chap. IV, \S9]{ATAEC} for instance), one can give a more precise description:

\begin{prop}\label{prop.badred}
Let $Z_{a, \gamma}$ be the set of places of $K$ that divide $\wp_a(t)\cdot(\wp_a(t)^2-16\gamma)$. 
Then $E_{a, \gamma}$ has good reduction outside $S= Z_{a, \gamma}\cup\{\infty\}$.
The reduction of $E_{a, \gamma}$ at places $v\in S$ is as follows:
\begin{center}\renewcommand{\arraystretch}{2.0}
\begin{tabular}{| c | c  | c | c |} 
\hline
Place $v$ of $K$&   Fiber of $E_{a, \gamma}$ at $v$  & $\ord_{v}\Delta_{\min}(E_{a, \gamma})$ & $\ord_{v}\cond(E_{a, \gamma})$  \\
\hline \hline
$v \mid \wp_a(t)$  & $\type{I}_{4}$   & $4$ & $1$  \\ \hline
$v \mid \wp_a(t)^2-16\gamma$ &  $\type{I}_2$      & $2$ &$1$ \\  \hline
$\infty$    & $\type{I}_{4q^a}$   &    $4q^a$ & $1$  \\ \hline
\end{tabular}
\end{center}
\renewcommand{\arraystretch}{1.0}  
\end{prop}
In this table, for all places  $v$   of bad reduction for $E_{a, \gamma}$, we have denoted by $\ord_{v}\Delta_{\min}(E_{a, \gamma})$ (resp. $\ord_v\cond(E_{a, \gamma})$) the valuation at $v$ of the minimal discriminant of $E_{a, \gamma}$ (resp. of the conductor of $E_{a, \gamma}$). 
See \cite[Chap. IV, \S9]{ATAEC}, \cite[Lect.\,1  \S8]{UlmerParkCity} for the definitions of these local invariants.

\paragraph{}
From this local information, one deduces the values of the following global invariants (we refer to \cite[Lect.\,1]{UlmerParkCity} for their definition). 
The minimal discriminant divisor $\Delta_{\min}(E_{a, \gamma})$ 
has degree ${\deg \Delta_{\min}(E_{a, \gamma}) = 12 q^{a}}$, 
and the conductor $\cond(E_{a, \gamma})\in\mathrm{Div}(\P^1)$ has degree $\deg \cond(E_{a, \gamma}) = 3q^a+1$.
Indeed, since both $\wp_a(t)$ and $\wp_a(t)^2-16\gamma$ are squarefree in $\F_q[t]$, one has 
\[ \sum_{v\mid \wp_a(t)} \deg v = \deg \wp_a(t) = q^a \quad\text{ and }\quad \sum_{v\mid \wp_a(t)^2-16\gamma} \deg v = \deg (\wp_a(t)^2-16\gamma) = 2q^a.\]
Hence the exponential differential height $H(E_{a, \gamma})$  is $q^{q^a}$ since, by definition (see \S2 in \cite[Lect.\,3]{UlmerParkCity}), it is given by $H(E_{a, \gamma}) := q^{(\deg \Delta_{\min}(E_{a, \gamma}))/12}$.
Summarising these calculations, we have 
\begin{equation}\label{eq.invariants}
 \deg \Delta_{\min}(E_{a, \gamma}) = 12 q^a, \quad
\deg \cond(E_{a, \gamma}) = 3q^a+1, \quad \text{ and } \quad 
H(E_{a, \gamma}) = q^{q^a}.
\end{equation}

\begin{rema}\label{rem.minimod.loc}
As is clear from \eqref{eq.disc} and Proposition \ref{prop.badred}, the discriminant $\Delta$ 
of the Weierstrass model \eqref{eq.Wmodel} 
has the same valuation as~${\Delta_{\min}(E_{a, \gamma})}$ at all finite places of $K$. 
Therefore, the model \eqref{eq.Wmodel} is a minimal integral model of $E_{a, \gamma}$ at all places $v\neq\infty$ of $K$. 
\end{rema}

\subsection[L-function, analytic rank and special value]{Definitions of $L$-function, analytic rank and special value}
\label{subsec.Lfunc.defi}

For any place $v$ of $K$, with degree $d_v$ and residue field $\F_v$, we denote by $(\widetilde{E_{a, \gamma}})_v$ the reduction modulo $v$ of a minimal integral model of $E_{a, \gamma}$ at $v$: $(\widetilde{E_{a, \gamma}})_v$ is thus a plane cubic curve over $\F_{v}$.
By definition, the $L$-function of $E_{a, \gamma}$ is the power series in $T$ given by
\begin{equation}\label{eq.def.Lfunc}
 L(E_{a, \gamma}, T) = \prod_{v\text{ good }} \left(1 - a_v \cdot T^{d_v} + q^{d_v} \cdot  T^{2d_v}\right)^{-1} \cdot
\prod_{v\text{ bad }} \left(1 - a_v \cdot  T^{d_v}  \right)^{-1} \in\Z[[T]],
\end{equation}
where the products are over places of $K$ where $E_{a, \gamma}$ has good (resp. bad) reduction, and where 
\begin{equation*}
a_v := \begin{cases}
q^{d_v} +1 - |(\widetilde{E_{a, \gamma}})_v(\F_{v})|
& \text{ if $E_{a, \gamma}$ has good reduction at }v, \\
0 & \text{ if $E_{a, \gamma}$ has additive reduction at }v, \\ 
+1& \text{ if $E_{a, \gamma}$ has split multiplicative reduction at }v, \\ 
-1& \text{ if $E_{a, \gamma}$ has nonsplit multiplicative reduction at }v. 
 \end{cases}
\end{equation*}
We refer to \cite[\S2.2]{BaigHall} and \cite[Lect.1  \S9, Lect.\,3  \S6]{UlmerParkCity} for more details. 

Since $E_{a, \gamma}$ is not isotrivial, a deep theorem of Grothendieck shows that $L(E_{a, \gamma}, T)$ is actually a polynomial in $T$ with integral coefficients whose degree is denoted by $b(E_{a, \gamma})$. 
Further, by the Grothendieck-Ogg-Shafarevich formula and our computation of the degree of $\cond(E_{a, \gamma})$ (see \eqref{eq.invariants}),    
we know that
\begin{equation}\label{eq.degL}
b(E_{a, \gamma}) =  \deg L(E_{a, \gamma}, T) = \deg \cond(E_{a, \gamma}) -4 = 3(q^a-1).
\end{equation}
In section \ref{sec.Lfunc}, we will compute the polynomial $L(E_{a, \gamma}, T)$ explicitly. For now, we only note that it makes sense to define the following two quantities: 

\begin{defi}\label{defi.spval}
Let $\rho(E_{a, \gamma})$ be the \emph{analytic rank of $E_{a, \gamma}$} \ie{}, the multiplicity of $T=q^{-1}$ as a root of~$L(E_{a, \gamma}, T)$. 
Further, define the \emph{special value of $L(E_{a, \gamma}, T)$ at $T=q^{-1}$} to be
\begin{equation}\label{eq.spval.def}
L^\ast(E_{a, \gamma},1) := \left.\frac{L(E_{a, \gamma}, T)}{(1-qT)^\rho}\right|_{T=q^{-1}}\in\Z[q^{-1}]\smallsetminus\{0\}, \quad \text{ where }\rho = \rho(E_{a, \gamma}).
\end{equation}
\end{defi}

\begin{rema}
The special value $L^\ast(E_{a, \gamma}, 1)$ is ``usually'' defined as the first nonzero coefficient in the Taylor expansion around $s=1$ of the function $s\mapsto L(E_{a, \gamma}, q^{-s})$. 
Our definition \eqref{eq.spval.def} differs from that more ``usual'' one by a factor $(\log q)^\rho$. 
We prefer to use the normalisation \eqref{eq.spval.def} because it ensures that $L^\ast(E_{a,\gamma}, 1)\in\Q^\ast$. 
This choice is consistent with our normalisation of $\Reg(E_{a, \gamma})$, see \S\ref{subsec.BSD} below.
\end{rema}

\subsection{The BSD conjecture}
\label{subsec.BSD}

The Mordell--Weil theorem implies that $E_{a, \gamma}(K)$ is a finitely generated abelian group (\cf{} \cite[Lect.\,1, Thm. 5.1]{UlmerParkCity}). 
Since the canonical N\'{e}ron--Tate height $\hhat{h}:E_{a, \gamma}(K) \to \Q$ is quadratic, it induces a $\Z$-bilinear pairing ${\langle \cdot,\cdot \rangle: E_{a, \gamma}(K)\times E_{a, \gamma}(K)\to\Q}$, which is nondegenerate modulo $E_{a, \gamma}(K)\tors$ (\cf{} \cite[Chap. III, Thm. 4.3]{ATAEC}).
We can then define the \emph{N\'{e}ron-Tate regulator} of $E_{a, \gamma}$ by 
\[\Reg(E_{a, \gamma}) := \left|\det\left( \langle P_i, P_j \rangle\right)_{1\leq i,j \leq r}\right| \in\Q^\ast,\]
for any choice of a $\Z$-basis $P_1, \dots, P_r \in E_{a, \gamma}(K)$ of $E_{a, \gamma}(K)/E_{a, \gamma}(K)\tors$. 
Note that we normalise $\langle \cdot,\cdot\rangle$ to have values in $\Q$: we may do so since, in our context, this height pairing has an interpretation as an intersection pairing on the minimal regular model of $E_{a, \gamma}$ (see \cite[Chap. III, \S9]{ATAEC}).    

Let us also recall that the \emph{Tate--Shafarevich group} of $E_{a, \gamma}/K$ is defined by
\[\sha(E_{a, \gamma}) :=
\ker\left( \H^1(K, E_{a, \gamma}) \longrightarrow \prod_{v} \H^1(K_v, (E_{a, \gamma})_v)\right),\]
see \cite[Lect.\,1 \S11]{UlmerParkCity} for more details. 
In Theorem \ref{theo.BSD} right below, we will see that $\sha(E_{a, \gamma})$ is finite. 
 
It has been conjectured by Birch, Swinnerton-Dyer and Tate that the ``analytic'' quantities $\rho(E_{a, \gamma})$ and $L^\ast(E_{a, \gamma}, 1)$ have an arithmetic interpretation (see \cite[Conj. B]{Tate_BSD}). 
Even though this conjecture is still open in general, it has been proved by Pries and Ulmer for $E_{a, \gamma}$ in \cite{UlmerPries}. 
Let us state their result as follows:

\begin{theo}[Pries - Ulmer]\label{theo.BSD}
For all $\gamma\in\F_q^\times$ and all integers $a\geq 1$, the elliptic curve $E_{a, \gamma}/K$ satisfies the full Birch and Swinnerton-Dyer conjecture. That is to say,
\begin{enumerate}[$\bullet$]
\item The Tate--Shafarevich group $\sha(E_{a, \gamma})$ is finite.
\item The rank of $E_{a, \gamma}(K)$ is equal to $\rho(E_{a, \gamma})=\ord_{T=q^{-1}}L(E_{a, \gamma}, T)$.
\item Moreover, one has
\begin{equation}\label{eq.BSD}
L^\ast(E_{a, \gamma}, 1) 
= \frac{|\sha(E_{a, \gamma})| \cdot \Reg(E_{a, \gamma})}{H(E_{a, \gamma})}\cdot \frac{\tam(E_{a, \gamma})\cdot q}{|E_{a, \gamma}(K)\tors|^2},
\end{equation}
where $\tam(E_{a, \gamma})$ denotes the Tamagawa number of $E_{a, \gamma}$.
\end{enumerate}
\end{theo}

    \begin{proof} 
    We only sketch a proof and refer the interested reader to \cite[\S3]{UlmerPries} for more details.
    As we have seen at the beginning of this section, $E_{a, \gamma}$ is $2$-isogenous to $E^\circ_{a, \gamma}$. 
    Since $E_{a, \gamma}$ and $E^\circ_{a, \gamma}$ are linked by an isogeny of degree prime to the characteristic of $K$, Theorem 7.3 in \cite[Chap. I]{Milne_ArithDual} implies that the BSD conjecture holds for $E_{a, \gamma}$ if and only if it does for $E^\circ_{a, \gamma}$. 
    Hence Theorem \ref{theo.BSD} will follow 
    if we prove 
    that $E^\circ_{a, \gamma}$ satisfies the BSD conjecture.
    
    We have also shown that $E^\circ_{a, \gamma}$ is birational to the curve $Z_{a, \gamma}\subset\P^1\times\P^1$ which,
    by construction,  is given in affine coordinates  by an equation of the form
    $f_\gamma(x)-f_\gamma(y)=\wp_a(t)$
    where $f_\gamma$ is a certain rational function on $\P^1$ over $K$ and $\wp_a(t)\in\F_q[t]$ is a separable additive polynomial. 
    Under these conditions, Corollary 3.1.4 of \cite{UlmerPries} proves that $E^\circ_{a, \gamma}$ satisfies the BSD conjecture. 
    
    The crucial point of the proof is the following:  given the specific shape of the equation of $Z_{a, \gamma}$ to which~$E^\circ_{a, \gamma}$ is birational, the minimal regular model $\Ecal^\circ_{a, \gamma}\to\P^1$ over $\F_q$ of the curve $E^\circ_{a, \gamma}/K$ is dominated by a product of curves $\Ccal_{a, \gamma}\times\Ccal_{a, \gamma}\dashrightarrow \Ecal^\circ_{a, \gamma}$ over $\F_q$ (where $\Ccal_{a, \gamma}$ is actually the curve defined in Remark \ref{rem.kloos.norm} below). 
    The Tate conjecture (T) asserts that the order of the pole of the zeta function of a surface $S/\F_q$ equals the rank of the N\'{e}ron--Severi group of $S$ (see \cite[Conj. C]{Tate_BSD}, or  \S10--\S13 in \cite[Lect.\,2]{UlmerParkCity}).
    Conjecture (T) is proved for surfaces that are dominated  by products of curves. 
    In particular, (T) holds for the surface $\Ecal^\circ_{a, \gamma}/\F_q$. 
    On the other hand, it is known that  conjecture (T) for  $\Ecal^\circ_{a, \gamma}/\F_q$ is equivalent to the BSD conjecture for the generic fiber of $\Ecal^\circ_{a, \gamma}\to\P^1$ \ie{}, for the elliptic  curve $E^\circ_{a, \gamma}/K$ (Theorem 8.1 in \cite[Lect.\,2]{UlmerParkCity}).
    Hence the result.
    \ProofEnd\end{proof}

In section \ref{sec.bnd.spval} we give bounds on the special value $L^\ast(E_{a, \gamma}, 1)$ on the left-hand side of \eqref{eq.BSD} and, in section~\ref{sec.BS}, 
we deduce from these an estimate on the asymptotically significant quantities on the right-hand side of \eqref{eq.BSD}.
For completeness, let us recall the following bounds (which we will need to prove Theorem \ref{theo.BS}).
\begin{prop} \label{prop.dioph.bnd}
Let $E$ be a nonisotrivial elliptic curve over $K$. Then
\begin{multicols}{2}
\begin{enumerate}[(i)]
\item $|E(K)\tors| \ll_q 1$,
\item $\log\tam(E) = o\big( \log H(E)\big)$ as $H(E)\to\infty$.
\end{enumerate}
\end{multicols}
\end{prop}

    \begin{proof} 
    The first bound is the analogue for elliptic curves over  $K=\F_q(t)$ of Merel's uniform bound on torsion for elliptic curves over $\Q$. 
    There are several proofs of \textit{(i)} and we refer the reader to \cite[Lect.\,I  \S7]{UlmerParkCity} for a survey and a sketch of proof (by a modular method). 
    The bound \textit{(ii)} on the Tamagawa number is a consequence of 
    Theorem~{1.22} of \cite{HP15} 
    in the case when $E$ is semistable or $p>3$.
    A self-contained (and elementary) proof for all elliptic curves over $K$ can also be found in \cite[Théorème 1.5.4]{Griffon_PHD}.
    \ProofEnd\end{proof}

\section{Characters and Kloosterman sums}
\label{sec.prelim}

The goal of this section is to introduce the objects which appear in the $L$-function of $E_{a, \gamma}$.

We fix a finite field field $\F_q$ of odd characteristic and  a nontrivial additive character $\psi_q$ on $\F_q$, which we assume to take values in the cyclotomic field $\Q(\zeta_p)$.
For instance, a standard choice of $\psi_q$ is the map 
$\psi_q: x\mapsto {\zeta_p}^{\trace_{\F_q/\F_p}(x)}$ where $\zeta_p$ is  a primitive $p$-th root of unity and $\trace_{\F_q/\F_p}:\F_q\to\F_p$ is the trace map.
 
For any finite extension $\F$ of $\F_q$, we denote by $\trace_{\F/\F_q} : \F\to\F_q$ the relative trace and we ``lift'' $\psi_q$ to a nontrivial additive character $\psi_\F:\F\to\Q(\zeta_p)^\times$ on $\F$ by putting $\psi_\F := \psi_q\circ\trace_{\F/\F_q}$.

\subsection{Kloosterman sums} 
\label{subsec.kloosterman}

For a finite field $\F$ of odd characteristic $p$, a nontrivial additive character $\psi$ on $\F$ with values in the cyclotomic field $\Q(\zeta_p)$, and a parameter $\alpha \in \F^\times$, we define the \emph{Kloosterman sum $\Kl_\F(\psi;\alpha)$} by:
\begin{equation}\label{eq.defi.kloo}
\Kl_\F(\psi; \alpha) := - \sum_{x\in \F^\times} \psi\left(x+\frac{\alpha}{x}\right).
\end{equation}
As a sum of $p$-th roots of unity, $\Kl_\F(\psi;\alpha)$ is an algebraic integer in $\Q(\zeta_p)$. 
For our purpose, it is convenient to normalise the sum by a $-1$ sign.
Let us gather in one proposition several classical facts about the Kloosterman sums that will be useful in this article.

\begin{prop}\label{prop.kloo}
 Let $\F, \psi$ and $\alpha$ be as above. Then:
\begin{enumerate}[{\rm (i)}]
\itemsep0.1em 
\item\label{kloo.item1} $\Kl_\F(\psi; \alpha)$ is a totally real algebraic integer in $\Q(\zeta_p)$ \ie{}, $\Kl_\F(\psi;\alpha)\in\Z[\zeta_p+\zeta_p^{-1}]$.

\item\label{kloo.item2} $\Kl_\F(\psi;\alpha)$ satisfies ``Sali\'{e}'s formula'': 
\begin{equation}\label{eq.saliesformula}
\Kl_\F(\psi; \alpha) = - \sum_{y\in \F} \lambda(y^2-4\alpha) \cdot\psi(y),
\end{equation}
where $\lambda:\F^\times\to\{\pm1\}$ is the unique multiplicative character on $\F^\times$ of exact order $2$ (extended by $\lambda(0):=0$ to the whole of $\F$).

\item\label{kloo.item3} If $\F$ contains $\F_q$, one has $\Kl_\F(\psi_q\circ\trace_{\F/\F_q};\alpha) = \Kl_\F(\psi_q\circ\trace_{\F/\F_q}; \alpha^{q})$.

\item\label{kloo.item4} 
There exist two algebraic integers $\kkappa_{\F}(\psi;\alpha)$ and ${\kkappa}'_{\F}(\psi;\alpha)$ such that for any finite extension $\F'/\F$, 
\begin{equation}\label{eq.HDformula}
\kkappa_\F(\psi; \alpha)\cdot \kkappa'_\F(\psi; \alpha) = |\F|\quad \text{ and }\quad 
\Kl_{\F'}(\psi\circ\trace_{\F'/\F};  \alpha) = \kkappa_{\F}(\psi; \alpha)^{[\F':\F]} + {\kkappa}'_{\F}(\psi; \alpha)^{[\F':\F]}.
\end{equation}
The pair $\{\kkappa_{\F}(\psi;\alpha), {\kkappa}'_{\F}(\psi;\alpha)\}$ is uniquely determined by $\F, \psi, \alpha$.
\item\label{kloo.item5} 
 $\kkappa_{\F}(\psi;\alpha)$ and $\kkappa'_{\F}(\psi;\alpha)$ have magnitude $|\F|^{1/2}$ in any complex embedding.  

\item\label{kloo.item7} In any complex embedding of $\Q(\zeta_p)$, one has
\begin{equation}\label{eq.nonvanishing}
0 < |\Kl_\F(\psi;\alpha)| < 2{|\F|}^{1/2}.
\end{equation}
\end{enumerate}

\end{prop}

\begin{proof}
The reader can confer \cite[Chap. 5, \S5]{LidlN}  and \cite[\S3]{VdGVdV} for proofs of these 
classical results about Kloosterman sums: \eqref{kloo.item1} and \eqref{kloo.item3} are easily checked; 
items \eqref{kloo.item2},  \eqref{kloo.item4} and \eqref{kloo.item5} 
are Theorems 5.47, 5.43 and 5.44 in  \cite{LidlN}, respectively; \eqref{kloo.item7} is proved in Corollary 3.2 of  \cite{VdGVdV}. 
\ProofEnd
\end{proof}

\subsection[Normalisation of Kloosterman sums]{The sums $\Kln_\gamma(v)$} 
\label{subsec.kloos.norm}

Assume a parameter $\gamma\in\F_q^\times$ is given. 
A place $v\neq 0, \infty$ of $K$ with degree $d_v$ 
corresponds to a monic irreducible polynomial $B_v \in\F_q[t]$ of degree $d_v$, with $B_v\neq t$. 
Choose a root $\beta_v\in\bar{\F_q}^\times$ of $B_v$: we claim that the value of the Kloosterman sum $\Kl_{\F_v}(\psi_{\F_v}; \gamma\beta_v^2)$ does not depend on the choice of $\beta_v$. 
Indeed, given one such $\beta_v$ the $d_v-1$ other choices are of the form ${\beta_v}^{q^{j}}$ (with $j\in\{1, 2, \dots, d_v-1\}$) because the $d_v$ different roots of $B_v$ in $\bar{\F_q}$ are all conjugate under the action of the Galois group $\Gal(\F_v/\F_q)$. 
A repeated application of  Proposition \ref{prop.kloo}\eqref{kloo.item3} proves the claim.
Therefore, the following definition makes sense:  

\begin{defi} 
Let $\F_q$ be a finite field of characteristic $p$, $\psi_q$ be a nontrivial additive character on $\F_q$ and~$\gamma\in\F_q^\times$. For any place $v\neq 0, \infty$ of $K=\F_q(t)$ corresponding to a monic irreducible $B_v\in\F_q[t]$, we let 
\begin{equation}\label{eq.defi.normkloo}
\Kln_{\gamma}(v) := \Kl_{\F_v}(\psi_{\F_v};\gamma \beta_v^2) = -\sum_{x\in\F_v^\times}\psi_q\circ\trace_{\F_v/\F_q}\left(x+\frac{\gamma\cdot \beta_v^2}{x}\right),
\end{equation}
for any choice of $\beta_v\in\bar{\F_q}^\times$ such that $B_v(\beta_v)=0$. 
\end{defi}
Note that $\Kln_\gamma(v)$ depends on $\F_q$ and $\psi_q$,  but we chose not to include these in the notation for brevity.

\paragraph{}
For any place $v\neq0, \infty$, Proposition \ref{prop.kloo}\eqref{kloo.item4}-\eqref{kloo.item5} shows that there exist a unique pair $\{\kkappan_\gamma(v),\kkappan'_\gamma(v)\}$ of conjugate algebraic integers, which have magnitude $|\F_v|^{1/2} = q^{d_v/2}$ in any complex embedding and such that
\begin{equation}\label{eq.defi.kkappa}
\Kln_\gamma(v) 
= \kkappan_\gamma(v) +\kkappan'_\gamma(v).
\end{equation}
In other words, we denote by $\{\kkappan_\gamma(v),\kkappan'_\gamma(v)\}$ the pair of algebraic integers $\{\kkappa_{\F_v}(\psi_{\F_v};\gamma\beta_v^2), \kkappa'_{\F_v}(\psi_{\F_v};\gamma\beta_v^2)\}$. 

\begin{rema} \label{rem.kloos.norm}
These sums $\Kln_\gamma(v)$ appear in the zeta function of a curve over $\F_q$. 
Namely, consider the hyperelliptic curve $\Ccal_{a, \gamma}$ over $\F_q$ defined as a smooth projective model of the affine curve $\wp_a(y)= x+\gamma/x$. 
A computation, which probably goes back to Weil, shows that the zeta function of $\Ccal_{a, \gamma}/\F_q$ is given by 
\[Z(\Ccal_{a, \gamma}/\F_q; U) = \frac{\prod_{v}\left(1- \kkappan_\gamma(v)\cdot U^{d_v}\right)\left(1- \kkappan'_\gamma(v)\cdot U^{d_v}\right) }{(1-U)(1-q\cdot U)},\]
where the product is over all places $v\neq 0, \infty$ of $K$ whose degrees divide $a$ (see \cite{VdGVdV}). 
\end{rema}
 
\subsection[Number of solutions of ``Artin-Schreier'' equations]{Number of solutions of ``Artin-Schreier'' equations}

In the process of computing the $L$-function of $E_{a, \gamma}$, we will need the following ``counting lemma'':

\begin{lemm}\label{lemm.solcount} 
Let $\F_q$ be a finite field and $\psi_q$ be a nontrivial additive character on $\F_q$. 
For any finite extension $\F$  of $\F_q$ and  any $z\in\F$, one has
\begin{equation}
\left|\{\tau\in\F : \wp_a(\tau)=z\}\right| 
=\sum_{\beta\in \F_{q^a}\cap \F} 
\psi_{\F}(\beta \cdot z).
\end{equation}
\end{lemm}

    \begin{proof}
     For the duration of the proof, we write $\F=\F_{q^n}$ and we let $b=\gcd(a,n)$; note that ${\F_{q^b}=\F_{q^a}\cap \F_{q^n}}$. 
     The map $\wp_a:\F_{q^n}\to\F_{q^n}$ is an $\F_{q^b}$-linear endomorphism of $\F_{q^n}$ whose kernel is $\F_{q^b}$. 
     In particular, the image of $\wp_a$ must have dimension $\dim_{\F_{q^b}}(\F_{q^n})-1=n/b-1$.
    The trace $t := \trace_{q^n/q^b} : \F_{q^n}\to\F_{q^b}$ is a surjective $\F_{q^b}$-linear map, so that its kernel $H$ is a sub-$\F_{q^b}$-vector space of $\F_{q^n}$ of dimension $n/b -1$. 
    Recall that $t\circ\wp_a = 0$ on $\F_{q^n}$, thus we have $\mathrm{Im}\,\wp_a\subset H$. 
    Since these two subspaces have the same dimension, they coincide. 
    This shows that, for $z\in\F_{q^n}$, 
    \[ \left|\{\tau\in\F_{q^n} : \wp_a(\tau)=z\}\right| 
    = \begin{cases}
    0 & \text{if } t(z)\neq0, \\ 
    |\ker\wp_a| = q^b   & \text{if } t(z)=0. 
    \end{cases} \]
    On the other hand, for a given $z\in\F_{q^n}$, we notice that
    \begin{align*}
    \sum_{\beta\in \F_{q^b}} \psi_{\F_{q^n}} (\beta \cdot z)
    &=\sum_{\beta\in \F_{q^b}} \psi_{\F_{q^b}}\circ \trace_{q^n/q^b} (\beta \cdot z)
    =\sum_{\beta\in \F_{q^b}} \psi_{\F_{q^b}}\circ\left( \beta \cdot \trace_{q^n/ q^b}(z)\right) 
    = \begin{cases}
    0 & \text{if } t(z)\neq0, \\ 
    q^b   & \text{if } t(z)=0, 
    \end{cases} 
    \end{align*}
    because $t=\trace_{q^n/q^b}$ is $\F_{q^b}$-linear. 
    Combining the two displayed equalities, we obtain the result.
    \ProofEnd\end{proof}

\subsection{Preliminary estimates}
\label{subsec.prelim}

For all $n\geq1$, we denote by $\pi_q(n)$ the number of places $v\neq 0, \infty$ of $K$ of degree $d_v=n$. 
Equivalently, $\pi_q(n)$ is the number of closed points of degree $n$ of the multiplicative group $\G_m$ over $\F_q$ \ie{}, the number of orbits of $\G_m(\bar{\F_q})$ under the action of $\Gal(\bar{\F_q}/\F_q)$ which have cardinality $n$.
We also let $P_q(n)$ be the set of closed points of~$\G_m$ whose degree divides $n$. 
Clearly, one has $|P_q(n)| = \sum_{d\mid n} \pi_q(d)$. 

In what follows, we will frequently need the following estimates, which we record here for convenience.

\begin{lemm}\label{lemm.estimates} Given a finite field $\F_q$, one has 
\begin{multicols}{2}
\begin{enumerate}[(i)]
\item\label{estim.item1} 
$q^n \ll_q n\cdot \pi_q(n) \ll_q q^n$ for all $n\geq 1$.
\item\label{estim.item2} 
$q^n \ll_q n\cdot |P_q(n)| \ll_q q^n$ for all $n\geq 1$.
\end{enumerate}
\end{multicols}
The involved constant are all effective and depend at most on $q$.
\end{lemm}

    \begin{proof} 
    The  ``Prime Number Theorem'' for $\F_q[t]$ directly implies \eqref{estim.item1}. 
    Being a bit more careful, one can even give explicit constants (as in \cite[Prop. 6.3]{Brumer} for example). 
    From this estimate and the relation between $\pi_q(n)$ and $|P_q(d)|$, one easily deduces \eqref{estim.item2}. 
    \ProofEnd\end{proof}

\section[The L-function]{The $L$-function}
\label{sec.Lfunc}   

For any integer $a\geq 1$, we denote by $P_q(a)$ the set of places $v$ of $K$, with $v\notin\{0, \infty\}$, whose degrees $d_v$   divide~$a$. 
 In the identification between finite places of $K$ and monic irreducible polynomials, $P_q(a)$ corresponds to $\left\{B\in\F_q[t] : B \text{ monic, irreducible s.t. } \deg B\mid a \text{ and }B\neq t \right\}$.
Equivalently, $P_q(a)$ is the set of closed points on $\G_m$ whose degree divides $a$.

 \paragraph{} We can now state the first main result of this article:
 
 \begin{theo}\label{theo.Lfunc} Let $\F_q$ be a finite field of odd characteristic $p$ and $K=\F_q(t)$. For any $\gamma\in\F_q^\times$ and any integer $a\geq 1$, consider the elliptic curve $E_{a, \gamma}/K$ given by \eqref{eq.Wmodel}. 
 The $L$-function of $E_{a, \gamma}$  is given by 
\begin{equation}\label{eq.Lfunc}
L(E_{a, \gamma}, T) 
= \prod_{v\in P_q(a)}\left(1-q^{d_v} \cdot T^{d_v}\right) \left(1-\kkappan_{\gamma}(v)^2 \cdot T^{d_v}\right)\left(1- \kkappan'_{\gamma}(v)^2 \cdot T^{d_v}\right), 
\end{equation}
 where 
 $\kkappan_{\gamma}(v)$, $\kkappan'_{\gamma}(v)$ are the algebraic integers associated to $\Kln_\gamma(v)$ (see \eqref{eq.defi.normkloo} 
 and \eqref{eq.defi.kkappa}).
 \end{theo}
 The proof of this theorem occupies the rest of the present section. 
Our strategy is loosely based on the computation in \cite[\S3.2]{Ulmer_legII}: to give an expression of $L(E_{a, \gamma}, T)$, 
we rely on an explicit ``point-counting'' argument with character sums.
This requires proving an identity between the character sums that appear in the argument and the Kloosterman sums introduced above. 
We first give in the next subsection a proof of this identity, and then prove Theorem \ref{theo.Lfunc} in \S\ref{subsec.Lfunc.proof}.

\begin{rema} Before proving this theorem, we note the following:
\begin{enumerate}[(1)]
\itemsep0.1em
\item Even though the sums $\Kln_\gamma(v)$ for $v\in P_q(a)$ depend on a choice of nontrivial additive character $\psi_q$ on $\F_q$, the $L$-function $L(E_{a, \gamma}, T)\in\Z[T]$ does not. 
Indeed, changing the choice of $\psi_q$ amounts to permuting the factors in \eqref{eq.Lfunc}.

\item Note that $\sum_{v\in P_q(a)} d_v = |\G_m(\F_{q^a})| = q^a-1$. 
 Thus, as a polynomial in $T$, the $L$-function $L(E_{a, \gamma}, T)$ has degree $3(q^a-1) = \deg \cond(E_{a, \gamma})-4$. This is consistent with the expected degree  (see \S\ref{subsec.Lfunc.defi}).


\item For any integer $a\geq 1$, the $L$-function of the base change of  $E_{a, \gamma}$ to $K_a:=\F_{q^a}(t)$ admits a somewhat simpler expression.
Indeed, for all $\gamma \in \F_q^\times$ and $a\geq 1$, the $L$-function of $E_{a, \gamma}/K_a$ is given by 
\begin{equation}
L(E_{a, \gamma}/K_a, T) 
= \prod_{\beta\in\F_{q^a}^\times}\left(1-q^{a} \cdot T\right) \left(1-(\kkappa_{\beta})^2 \cdot T\right)\left(1- (\kkappa'_{\beta})^2 \cdot T\right), 
\end{equation}
 where, for all $\beta\in\F_{q^a}^\times$,   $\kkappa_{\beta}$ and  $\kkappa'_{\beta}$ are the two algebraic integers associated to the Kloosterman sum~$\Kl_{\F_{q^a}}(\psi_{\F_{q^a}};\gamma\beta^2)$.  
 This follows directly from \eqref{eq.Lfunc}.

\item Recall the elliptic curves $E^\circ_{a, \gamma}$ introduced at the beginning of \S\ref{sec.family} and given by \eqref{eq.Wmod.X}. 
Since isogenous elliptic curves share the same $L$-function (by \cite[Chap. I, Lemma 7.1]{Milne_ArithDual}), Theorem \ref{theo.Lfunc} also shows that the $L$-function of $E^\circ_{a, \gamma}$ is given by \eqref{eq.Lfunc}. 

\item It is perhaps illuminating to comment on the appearance of Kloosterman sums in $L(E_{a, \gamma}, T)$. 
Choosing a prime $\ell\neq p$, we denote by $\H^i(X)=\H^i_{\text{\'{e}t}}(X \times_{\F_q} \bar{\F_q}, \Q_\ell)$ the $i$-th $\ell$-adic \'{e}tale cohomology group of a smooth projective variety $X/\F_q$.
Let $\Ecal^\circ_{a, \gamma}\to\P^1$ denote the minimal regular model of $E^\circ_{a, \gamma}$. 
By Grothendieck cohomological interpretation of $L$-functions, $L(E^\circ_{a, \gamma}, T)$ is  essentially the ``interesting part'' of the zeta function of the surface $\Ecal_{a, \gamma}/\F_q$ \ie{}, $L(E^\circ_{a, \gamma}, T)$ is a factor of the characteristic  polynomial of the Frobenius $\frob_q$ acting on $H^2(\Ecal^\circ_{a, \gamma})$.

By the construction of $E^\circ_{a, \gamma}$ in \S\ref{sec.family}, $\Ecal^\circ_{a, \gamma}$ is a smooth model of a quotient of $\Ccal_{a, \gamma} \times \Ccal_{a, \gamma}$ by the action of a certain finite group $G_a$, where $\Ccal_{a, \gamma}/\F_q$ is the curve introduced in  Remark \ref{rem.kloos.norm}   (see \cite[\S7.3]{UlmerPries} for a more detailed presentation). 
In particular, $\H^2(\Ecal^\circ_{a, \gamma})$ can be seen as a subspace of $\H^2((\Ccal_{a, \gamma} \times \Ccal_{a, \gamma})/G_a)$, itself a subspace of $\H^2(\Ccal_{a, \gamma} \times \Ccal_{a, \gamma})$. 
Hence, by Künneth's formula, $L(E^\circ_{a, \gamma}, T)$ divides the characteristic  polynomial of $\frob_q$ acting on $\H^1(\Ccal_{a, \gamma}) \otimes \H^1(\Ccal_{a, \gamma})$.
 
As was noted in Remark \ref{rem.kloos.norm}, the numerator of the zeta function of $\Ccal_{a, \gamma}$ \ie{}, the characteristic polynomial of $\frob_q$ acting on $\H^1(\Ccal_{a, \gamma})$, involves Kloosterman sums $\Kln_\gamma(v)$.
Hence, the eigenvalues of $\frob_q$ acting on $\H^1(\Ccal_{a, \gamma}) \otimes \H^1(\Ccal_{a, \gamma})$ are products of the form $\Kln_\gamma(v_1)\Kln_\gamma(v_2)$. 
Being a factor of the characteristic polynomial of $\frob_q$ acting on $\H^1(\Ccal_{a, \gamma}) \otimes \H^1(\Ccal_{a, \gamma})$, the $L$-function $L(E^\circ_{a, \gamma}, T)$ has to involve some of the products $\Kln_\gamma(v_1)\Kln_\gamma(v_2)$.

Working out the details of this 
sketchy computation could lead to a different proof of Theorem~\ref{theo.Lfunc}.

\end{enumerate}
\end{rema}

\subsection{An identity between character sums}
\label{subsec.id.charsum}

 For any finite finite field $\F$ of odd characteristic, we fix a nontrivial additive character $\psi$ on $\F$, and we denote by $\lambda:\F^\times\to\{\pm1\}$ the unique quadratic character on $\F^\times$, extended to the whole of $\F$ by $\lambda(0):=0$.  

 For any~$\gamma\in \F^\times$ and $\beta\in \F$, we consider the following double character sum
 \[M_{\F}(\beta, \gamma) := \sum_{x\in \F}\sum_{z\in \F} 
  \lambda\left( x(x+16\gamma)(x+z^2)\right)\cdot\psi(\beta z),\]
  Up to a trivial term, we identify the character sum $M_{\F}(\beta, \gamma)$ as the square of a Kloosterman sum:
\begin{prop} 
\label{prop.id.charsum}
 Notations being as above, we have
 \begin{equation}
 M_\F(\beta, \gamma) =
\begin{cases}
1 & \text{if }\beta=0 \\
\Kl_{\F}\big(\psi ; \gamma\cdot \beta^2\big)^2-|\F|
& \text{if }\beta\neq0.
\end{cases}
  \end{equation}
\end{prop}

In order to prove this identity, we begin by recording the following ``point-counting'' lemma:

\begin{lemm}\label{lemm.id.prelim} For any $z\in \F$ and $\gamma\in \F^\times$, consider 
\[ X_{\F, \gamma}(z) :=\left\{ (u,v)\in \F^2 : v^2 - (u^2-uz-4\gamma) v + \gamma z^2 = 0\right\} \subset \F^2.\]
We have
\begin{equation}\label{eq.id.prelim}
\left| X_{\F, \gamma}(z) \right| = |\F|\cdot (1+\delta_{z,0})- 1  + \sum_{x\in \F}\lambda\left(x(x+16\gamma)(x+z^2)\right),
\end{equation}
where $\delta_{z,0}=1$ if $z=0$ and $\delta_{z,0}=0$ otherwise.
\end{lemm}
  
    \begin{proof} We split the set $X_{\F, \gamma}(z) $ into two disjoint parts $X_0:=\left\{(u,v)\in X_{\F, \gamma}(z)  : v=0\right\}$ and $X_1 := X_{\F, \gamma}(z) \smallsetminus X_0$.
    Computing $|X_0|$ is straightforward: if $z=0$, any pair $(u,0)$ with $u\in\F$  belongs to $X_0$ and, if $z\neq0 $, no such pair belongs to $X_{\F, \gamma}(z)$. 
    Therefore, we have $|X_0| = |\F|\cdot \delta_{z,0}$.
    
    To count the number of elements in $X_1$, let us introduce an auxiliary set
    \[ Y_{\F, \gamma}(z) := \left\{(x,y)\in\F^2 : y^2=x(x+16\gamma)(x+z^2)\right\} \subset \F^2,\]
    which we also split into two parts: 
    $Y_0:=\left\{ (x,y)\in Y_{\F, \gamma}(z): x=0\right\}$ and $ Y_1 := Y_{\F, \gamma}(z) \smallsetminus Y_0$.
    The two maps
    \[ \begin{array}{cccc c cccc}
     & X_1 & \longrightarrow & Y_1 & \quad \text{ and } \quad \phantom{.}&  
      & Y_1 & \longrightarrow & X_1 \\ 
        & (u,v) & \longmapsto & \left(4v, 4v(2u-z)\right) & &
        & (x,y) & \longmapsto & \displaystyle \left(\frac{y+z x}{2x}, \frac{x}{4}\right).
     \end{array} \]
    are easily checked to be well-defined and inverse to each other. Thus, we have $|X_1|=|Y_1|$ and the Lemma will be proved once we have related $|Y_1|$ to the character sum in \eqref{eq.id.prelim}. Grouping points $(x,y)\in Y_1$ according to their $x$-coordinates, we obtain that:
    \[|Y_1|
    = \sum_{x\in \F^\times} \left( 1+ \lambda\big(x(x+16\gamma)(x+z^2)\big) \right)\\
    = |\F| -1 + \sum_{x\in \F}  \lambda\big(x(x+16\gamma)(x+z^2)\big).\]
    Since $\left|X_{\F, \gamma}(z)\right|  =   |X_0| + |Y_1|$, we conclude that relation \eqref{eq.id.prelim} holds.
    \ProofEnd\end{proof}

    \begin{proof}[of Proposition \ref{prop.id.charsum}] 
    We treat the case where $\beta=0$ first: we have to prove that
    \[\sum_{x\in \F}\sum_{z\in \F}   \lambda\left( x(x+16\gamma)(x+z^2)\right) = 1.\]
    This identity follows easily upon using the following equality (see \cite[Thm. 5.48]{LidlN}):
    \[\forall b,c\in \F, \quad \sum_{y\in\F} \lambda(y^2 + b y +c) = \begin{cases}
    |\F|-1 & \text{if } b^2=4c, \\ 
     -1 & \text{otherwise.} \end{cases}
     \]
    It remains to prove the Proposition in the case where $\beta\neq 0$.
    If $\beta\neq 0$, the character $x\mapsto\psi(\beta x)$ is nontrivial, and we can
    expand $\Kl_{\F}\big(\psi, \gamma \beta^2\big)^2$ as a double character sum  using Sali\'{e}'s formula \eqref{eq.saliesformula}: we obtain that
    \begin{align}
    \Kl_{\F}\big(\psi, \gamma \beta^2\big)^2
    &= \left(\sum_{u\in \F}\lambda(u^2-4\gamma)\cdot\psi(\beta u)\right)^2
    = \sum_{u_1\in \F}\sum_{u_2\in \F}\lambda\left((u_1^2-4\gamma)(u_2^2-4\gamma)\right)\cdot\psi\big(\beta(u_1+u_2)\big) \notag \\ 
    &= \sum_{z\in \F} \left(\sum_{u\in \F}\lambda\left((u^2-4\gamma)((u-z)^2-4\gamma)\right)\right)\cdot\psi(\beta z).\label{eq.id.inter1}
    \end{align}
    The last line follows from the first by letting $(u, z) =(u_1, u_1+u_2)$.  
    For a given $z\in\F$, notice that
    \[ \forall u \in \F, \qquad (u^2-4\gamma)((u-z)^2-4\gamma) = (u^2-zu-4\gamma)^2 - 4\gamma z^2,  \]
    which is the discriminant of the quadratic equation
    $V^2 -(u^2-zu-4\gamma)\cdot V + \gamma z^2=0$
    with unknown $V$. Hence, the number of solutions $v\in \F$ to this equation is given by
    \[  \left|\left\{ v\in \F  : v^2 -(u^2-zu-4\gamma)\cdot v +\gamma z^2=0 \right\}\right| 
    = 1 +\lambda\left( (u^2-zu-4\gamma)^2 - 4\gamma z^2\right).\]
    Therefore we can  rewrite the inner sum (over $u\in \F$) in \eqref{eq.id.inter1} as
     \begin{align*}
    \sum_{u\in \F}\lambda\left((u^2-4\gamma)((u-z)^2-4\gamma)\right)
    &= -|\F|+\sum_{u\in \F} \left(1+ \lambda\left( (u^2-zu-4\gamma)^2 - 4\gamma z^2\right) \right) \\
    &= - |\F| +\sum_{u \in \F } \left| \left\{ v\in \F  : v^2 -(u^2-zu-4\gamma)\cdot v +\gamma z^2=0 \right\}\right| \\
    &= -|\F|+\left| X_{\F, \gamma}(z)\right|,
    \end{align*}
    where $ X_{\F, \gamma}(z)$ is the set introduced in Lemma \ref{lemm.id.prelim}. In that same Lemma \ref{lemm.id.prelim}, we have proved that 
    \[\forall z\in \F, \qquad \left| X_{\F, \gamma}(z)\right| - |\F|  = |\F|\cdot \delta_{z,0}- 1  + \sum_{x\in \F}\lambda\left(x(x+16\gamma)(x+z^2)\right).\]
    On multiplying this identity by $\psi(\beta z)$ and summing over all $z\in \F$, we deduce from \eqref{eq.id.inter1} that
    \[\Kl_\F\big(\psi; \gamma\beta^2\big)^2 
    = \sum_{z\in \F} (|\F|\cdot \delta_{z,0}- 1)\psi(\beta z) 
    + M_\F(\beta, \gamma) 
    = |\F| + M_\F(\beta, \gamma).\]
    Indeed, the sum $\sum_{z\in \F} \psi(\beta z)$ vanishes because $x\mapsto\psi(\beta x)$ is nontrivial. This conclude the proof. 
    \ProofEnd\end{proof}
    
\subsection{Proof of Theorem \ref{theo.Lfunc}}
\label{subsec.Lfunc.proof}
 
We start by giving a ``concrete'' expression of $L(E_{a, \gamma}, T)$ in terms of the number of rational points on the reductions of $E_{a, \gamma}$ at places of $K$. To that end, we introduce the following notations: for any $n\geq 1$ and any~$\tau\in\P^1(\F_{q^n})$, denote by $v_\tau$ the place of $K$ corresponding to $\tau$ and by $(\widetilde{E_{a, \gamma}})_\tau$ the reduction of a integral minimal model of $E_{a, \gamma}$ at $v_\tau$ (a cubic plane curve over $\F_{{v_\tau}}$, not necessarily smooth). 
We then let
\[ A_{a, \gamma}(\tau, q^n) := q^n+1 - |(\widetilde{E_{a, \gamma}})_\tau(\F_{q^n})|.\]
With these notations, we have:
\begin{lemm}
The $L$-function of $E_{a, \gamma}$ is given by 
\begin{equation}\label{eq.practical.Lfunc}
\log L(E_{a, \gamma}, T) = \sum_{n=1}^{\infty} \left(\sum_{\tau\in\P^1(\F_{q^n})}A_{a, \gamma}(\tau, q^n)\right) \cdot \frac{T^n}{n}.
\end{equation}
\end{lemm}

    \begin{proof} Starting from the definition \eqref{eq.def.Lfunc} of the $L$-function of $E_{a, \gamma}$,  expanding $\log L(E_{a, \gamma}/K, T)$ as a power series in $T$ and rearranging terms yields the desired  expression of $L(E_{a, \gamma}, T)$. 
    
    See \cite[\S3.2]{Ulmer_legII} or \cite[\S2.2]{BaigHall} for more details. 
    \ProofEnd\end{proof}

The next step is to find a more tractable expression of the inner sums in \eqref{eq.practical.Lfunc}.
For any finite extension $\F_{q^n}$ of $\F_q$, we again let $\lambda :\F_{q^n}^\times\to\{\pm1\}$ 
be the unique nontrivial character of $\F_{q^n}^\times$ of order $2$. 
For any~$\tau\in\P^1(\F_{q^n})$, we choose an affine model $y^2=f_\tau(x)$ of $(\widetilde{E_{a, \gamma}})_\tau$ 
(with $f_\tau(x)\in\F_{v_{\tau}}[x]$ of degree~$3$): 
a standard computation then yields that 
\begin{equation}\label{eq.Atau}
A_{a, \gamma}(\tau, q^n)  
= q^n+1 - |\widetilde{(E_{a, \gamma})}_\tau(\F_{q^n})| 
= q^n - \sum_{x\in\F_{q^n}} \left(1 + \lambda\left( f_\tau(x)\right) \right) 
=- \sum_{x\in\F_{q^n}}  \lambda\left( f_\tau(x)\right).  
\end{equation}
Since $E_{a, \gamma}$ has split multiplicative reduction at $\infty$ (see Proposition \ref{prop.badred}), we have $A_{a, \gamma}(\infty, q^n) = 1$.
Moreover, by Remark \ref{rem.minimod.loc},  we may choose $f_\tau(x) = x(x+16\gamma)(x+\wp_a(\tau)^2)$ for any $\tau\in\F_{q^n}$.
Summing identity 
\eqref{eq.Atau} over all $\tau\in\F_{q^n}$ for this choice of $f_\tau(x)$, we obtain that
\[- \sum_{\tau \in\F_{q^n}}A_{a, \gamma}(\tau, q^n)
=  \sum_{x\in\F_{q^n}} \sum_{\tau\in\F_{q^n}} \lambda\left( x(x+16\gamma)(x+\wp_a(\tau)^2)\right). \]
We know from Lemma \ref{lemm.solcount} 
that, for any $z\in\F_{q^n}$, 
\begin{equation}\notag
\left|\{\tau\in\F_{q^n} : \wp_a(\tau)=z\}\right| 
=\sum_{\beta\in \F_{q^a}\cap \F_{q^n}} \psi_{q^n} 
(\beta \cdot z).
\end{equation}
Therefore, we deduce that
\begin{align*}
- \sum_{\tau \in\F_{q^n}}A_{a, \gamma}(\tau, q^n)
&=  \sum_{x\in\F_{q^n}} \sum_{z\in\F_{q^n}}  \left|\{\tau\in\F_{q^n} : \wp_a(\tau)=z\}\right| \cdot\lambda\left( x(x+16\gamma)(x+z^2)\right) \\
&= \sum_{\beta\in \F_{q^a}\cap \F_{q^n}}\left( \sum_{x\in\F_{q^n}} \sum_{z\in\F_{q^n}} \lambda\left( x(x+16\gamma)(x+z^2)\right)\cdot \psi_{q^n}(\beta \cdot z) \right)
= 
 \sum_{\beta\in \F_{q^a}\cap \F_{q^n}} M_{q^n}(\beta, \gamma),
\end{align*}
where $M_{q^n}(\beta, \gamma)$ is the double character sum that we studied in subsection \ref{subsec.id.charsum} (with $\F=\F_{q^n}$).
Proposition~\ref{prop.id.charsum}, together with our computation of $A_{a, \gamma}(\infty, q^n)$, then leads to
 \begin{equation}\label{eq.Lfunc.inter1}
\sum_{\tau \in\P^1(\F_{q^n})}A_{a, \gamma}(\tau, q^n)
= - \sum_{\beta\in(\F_{q^a}\cap\F_{q^n})\smallsetminus\{0\}} \left(\Kl_{q^n}\big(\psi_{q^n}, \gamma\beta^2\big)^2 - q^n\right). 
 \end{equation}

For each $\beta\in\F_{q^n}^\times$, Proposition \ref{prop.kloo}\eqref{kloo.item4} proves the existence of two algebraic integers $\kkappa_{\F_{q^n}}(\psi_{q^n}; \gamma\beta^2)$ and $\kkappa'_{\F_{q^n}}(\psi_{q^n}; \gamma\beta^2)$ whose product is $|\F_{q^n}|=q^n$ and whose sum is $\Kl_{q^n}(\psi_{q^n}; \gamma\beta^2)$. We thus have 
\begin{equation*}
\Kl_{q^n}\big(\psi_{q^n}, \gamma\beta^2\big)^2 - q^n 
= \left(\kkappa_{q^n}\big(\psi_{q^n}; \gamma\beta^2\big) + \kkappa'_{q^n}\big(\psi_{q^n}; \gamma\beta^2\big)\right)^2 -q^n
= \kkappa_{q^n}\big(\psi_{q^n}; \gamma\beta^2\big)^2 + \kkappa'_{q^n}\big(\psi_{q^n}; \gamma\beta^2\big)^2 + q^n.
\end{equation*}
Hence, for all $n\geq 1$, \eqref{eq.Lfunc.inter1} can be rewritten as  \begin{equation}\label{eq.Lfunc.inter2}
\sum_{\tau \in\P^1(\F_{q^n})}A_{a, \gamma}(\tau, q^n)
= - \sum_{\beta\in(\F_{q^a}\cap\F_{q^n})\smallsetminus\{0\}} \left(\kkappa_{q^n}\big(\psi_{q^n}; \gamma\beta^2\big)^2 + \kkappa'_{q^n}\big(\psi_{q^n}; \gamma\beta^2\big)^2 + q^n\right) 
 \end{equation}

For any $\beta\in\F_{q^a}^\times$, we let $d_\beta := [\F_{q}(\beta):\F_q]$ be the degree of $\beta$ over $\F_q$. Such a $\beta$ also belongs to $\F_{q^n}$ if and only if $d_\beta$ divides $n$ \ie{}, if and only if $\F_{q^n}$ is an extension of $\F_{q^{d_\beta}}$. Therefore, by Proposition \ref{prop.kloo}\eqref{kloo.item4} one has
\[ \Kl_{q^n}(\psi_{q^n}, \gamma\beta^2)
= \Kl_{q^n}\left(\psi_{q^{d_\beta}} \circ\trace_{q^n/q^{d_\beta}};\gamma\beta^2\right)
= \kkappa_{q^{d_\beta}}(\psi_{q^{d_\beta}}; \gamma\beta^2)^{n/d_\beta} + \kkappa'_{q^{d_\beta}}(\psi_{q^{d_\beta}}; \gamma\beta^2)^{n/d_\beta}.\]
For brevity, we temporarily write 
$\oomega_1(\beta) := \kkappa_{q^{d_\beta}}\big(\psi_{q^{d_\beta}}; \gamma\beta^2\big)^{2}$, 
$\oomega_2(\beta) := \kkappa'_{q^{d_\beta}}\big(\psi_{q^{d_\beta}}; \gamma\beta^2\big)^{2}$ and
${\oomega_3(\beta) := q^{d_\beta}}$.  
We deduce from \eqref{eq.Lfunc.inter2} that
\[
\sum_{\tau \in\P^1(\F_{q^n})}A_{a, \gamma}(\tau, q^n)
= - \sum_{\beta\in(\F_{q^a}\cap\F_{q^n})\smallsetminus\{0\}} \left(\oomega_1(\beta)^{n/d_\beta} + \oomega_2(\beta)^{n/d_\beta} +\oomega_3(\beta)^{n/d_\beta}\right). 
\]
Upon multiplying both sides of this identity by $T^n/n$ and summing over all integers $n\geq1 $, we arrive at
\begin{equation*}
- \log L(E_{a, \gamma}, T)
=   \sum_{n\geq 1}\left(\sum_{\beta\in(\F_{q^a}\cap\F_{q^n})\smallsetminus\{0\}} \left(\oomega_1(\beta)^{n/d_\beta} + \oomega_2(\beta)^{n/d_\beta} +\oomega_3(\beta)^{n/d_\beta}\right)\right)\cdot \frac{T^n}{n},
\end{equation*}
and exchanging the order of summation leads to
\begin{align}
- \log L(E_{a, \gamma}, T)
&= \sum_{\beta\in \F_{q^a}^\times}\left(
\sum_{m\geq 1} \left(\oomega_1(\beta)^{m} + \oomega_2(\beta)^{m} +\oomega_3(\beta)^{m}\right)\cdot \frac{T^{m d_\beta}}{m d_\beta}\right)\notag \\ 
&= \sum_{\beta\in \F_{q^a}^\times} - \frac{1}{d_\beta}
 \cdot  \log\left(\prod_{i=1}^3 \left(1-\oomega_i(\beta)\cdot T^{d_\beta}\right) \right)\label{eq.Lfunc.inter3}
 \end{align}

Moreover, as we have explained in \S\ref{subsec.kloos.norm}, for each $\beta\in\F_{q^a}^\times$ with degree $d_\beta$,  the triple $\{\oomega_1(\beta), \oomega_2(\beta), \oomega_3(\beta)\}$
 is constant along the Galois orbit $\{\beta, \beta^q, \dots, \beta^{q^{d_\beta-1}}\}$ \ie{}, the closed point of $\G_m$ corresponding to $\beta$. Consequently, 
for a closed point $v$ of $\G_m$ whose degree divides $a$, we may define $\oomega_i(v)$ to be $\oomega_i(\beta)$ for any choice of $\beta\in v$.
Each term $\log(1-\oomega_i(v)T^{d_v})$ (for $i=1, 2, 3$) appears $d_v=d_\beta$ times in \eqref{eq.Lfunc.inter3} and we may thus rewrite the sum over $\beta\in\F_{q^a}^\times$ there as a sum over closed points of $\G_m$ whose degrees divide $a$:
\[ \log L(E_{a, \gamma}, T)
=  \sum_{ v\in P_q(a)}  \log\left(\prod_{i=1}^3 \left(1-\oomega_i(v)\cdot T^{d_v}\right) \right) 
\]
Exponentiating this identity and replacing the $\oomega_i(v)$'s by their value concludes the calculation of the {$L$-function} of $E_{a, \gamma}$ over $K$. 
\ProofEnd 
 
 \subsection{Rank and special value}
 
 From the factored expression of $L(E_{a, \gamma}, T)$ obtained in Theorem \ref{theo.Lfunc}, one can deduce explicit expressions of the analytic rank $\rho(E_{a, \gamma})$ and of the special value $L^\ast(E_{a, \gamma}, 1)$ (as defined in \S\ref{subsec.Lfunc.defi}). 

 \begin{prop}\label{prop.anal.rk} 
 For any $\gamma\in\F_q^\times$ and any integer $a\geq 1$, we have
 \begin{align}
\rho(E_{a, \gamma}) &
 = |P_q(a)| \label{eq.anal.rk} \\
\text{and }\quad  L^\ast(E_{a, \gamma}, 1) 
&= \prod_{v\in P_q(a)}4\deg v \cdot \prod_{v\in P_q(a)} \left(1 -  \frac{\Kln_{\gamma}(v)^2}{4q^{\deg v}}\right).\label{eq.expr.spval} 
\end{align}
 \end{prop}
  
    \begin{proof} 
    For each $v\in P_q(a)$, let 
    \[ g_v(T) := \left(1-q^{d_v} \cdot T^{d_v}\right) \left(1-\kkappan_{\gamma}(v)^2 \cdot T^{d_v}\right)\left(1- \kkappan'_{\gamma}(v)^2 \cdot T^{d_v}\right)\] 
    be the corresponding factor of $L(E_{a, \gamma}, T)$ (as in Theorem \ref{theo.Lfunc}).
    The factor $1-q^{d_v}\cdot T^{d_v}$ has a simple zero at $T=q^{-1}$  and neither of the other two factors of $g_v(T)$ vanishes at $T=q^{-1}$. Indeed, neither of $\kkappan_{\gamma}(v)^2$ and $\kkappan'_{\gamma}(v)^2$ can equal $q^{d_v}$ 
    since $|\Kln_{\gamma}(v)| = |\kkappan_{\gamma}(v)+\kkappan'_{\gamma}(v)|$ is strictly smaller than $2 q^{d_v/2}$ by Proposition \ref{prop.kloo}\eqref{kloo.item7}.
    Hence, $g_v(T)$ has a simple zero at $T=q^{-1}$ and we have 
    \[\rho(E_{a, \gamma}) = \ord_{T=q^{-1}}L(E_{a, \gamma}, T) 
    = \sum_{v\in P_q(a)} \ord_{T=q^{-1}}g_v(T) 
    = |P_q(a)|,\]
    which proves \eqref{eq.anal.rk}. 
    Next, by definition (\cf{} \eqref{eq.spval.def}) of the special value $L^\ast(E_{a, \gamma}, 1)$, we have
    \[ L^\ast(E_{a, \gamma}, 1)
    =  \left.\frac{L(E_{a, \gamma}, T)}{(1-qT)^\rho}\right|_{T=q^{-1}}
    = \prod_{v\in P_q(a)}  \left(\!\left.\frac{g_v(T)}{1-qT}\right|_{T=q^{-1}}\right).    \]
    Furthermore, by a straightforward computation, we obtain that
    \[\left.\frac{g_v(T)}{1-qT}\right|_{T=q^{-1}}
    = d_v \cdot \left( 1 - \frac{\kkappan_{\gamma}(v)^2}{q^{d_v}} \right)
    \left( 1 - \frac{\kkappan'_{\gamma}(v)^2}{q^{d_v}} \right)
    = d_v \cdot \frac{4 q^{d_v} - \Kln_{\gamma}(v)^2}{q^{d_v}}.
    \]
    Combining the last two displayed identities directly leads to the desired expression of $L^\ast(E_{a, \gamma}, 1)$. 
    \ProofEnd \end{proof}

Since the curve $E_{a, \gamma}$ satisfies the BSD conjecture (see Theorem \ref{theo.BSD}), Proposition \ref{prop.anal.rk} implies that $\rk E_{a, \gamma}(K)= |P_q(a)|$. 
It is worthwhile to note that the Mordell-Weil rank of $E_{a, \gamma}(K)$ is actually independent of $\gamma\in\F_q^\times$ (this is evident from \eqref{eq.anal.rk}).
 By Lemma \ref{lemm.estimates}\eqref{estim.item2}, we therefore have
\begin{equation}\label{eq.unbnd.rk}
\frac{q^a}{a} 
\ll_q \rk E_{a, \gamma}(K) \ll_q
\frac{q^a}{a}. 
\end{equation}
 In particular, we retrieve the following ``unbounded rank'' result (see \cite[Coro. 2.7.3]{UlmerPries}). 
\begin{coro}\label{coro.unbnd.rk}
Let $\F_q$ be a finite field of odd characteristic $p$ and $K=\F_q(t)$.
For all $\gamma\in\F_q^\times$,  
the rank of $E_{a, \gamma}(K)$ is unbounded as $a\geq 1$ tends to infinity.
\end{coro}
One might further compare \eqref{eq.unbnd.rk} to the upper bound of 
\cite[Prop. 6.9]{Brumer}, which states that 
\[\rk E_{a, \gamma}(K) \ll_q \frac{\deg\cond(E_{a, \gamma})}{\log \deg\cond(E_{a, \gamma})}.\]
From \S\ref{subsec.invariants}, we know that $q^a\ll \deg \cond(E_{a, \gamma})\ll q^a$; therefore the lower bound in \eqref{eq.unbnd.rk} shows that $\rk E_{a, \gamma}(K)$ attains Brumer's upper bound 
(up to constants depending on $q$).

\subsection[Angles of Kloosterman sums]{Angles of the sums $\Kln_\gamma(v)$} 
\label{subsec.anglesnorm}

For any $v\in P_q(a)$, 
we know that $\Kln_{\gamma}(v)$ is a 
real algebraic integer with $|\Kln_{\gamma}(v)|\leq 2q^{d_v/2}$ in any complex embedding (see items \eqref{kloo.item1}, \eqref{kloo.item4}, \eqref{kloo.item5} in Proposition \ref{prop.kloo}). 
Thus, there exists a unique angle $\angn_{\gamma}(v)\in [0, \pi]$ such that
\begin{equation}\label{eq.defi.angle}
\Kln_\gamma(v) = 2q^{d_v/2}\cdot\cos \angn_\gamma(v).
\end{equation}
Note that the individual angles $\angn_\gamma(v)$ depend on a choice of complex embedding $\Q(\zeta_p)\into\C$, 
but that the set $\{\angn_\gamma(v)\}_{v\in P_q(a)}$ does not. 

We can then rewrite the expression of $L^\ast(E_{a, \gamma}, 1)$ obtained in Proposition \ref{prop.anal.rk} in terms of these angles:
\begin{equation}\label{eq.expr.spval.alt} 
L^\ast(E_{a, \gamma}, 1) 
= \prod_{v\in P_q(a)}4\deg v \cdot \prod_{v\in P_q(a)}  \sin^2\big(\angn_\gamma(v)\big).
\end{equation}
In section \ref{sec.bnd.spval}, we will prove upper and lower bounds on the size of $L^\ast(E_{a, \gamma}, 1)$ in terms of the degree $b(E_{a, \gamma})$ of the $L$-function. 
From the above expression, it is obvious that this size crucially depends on how the angles $\left\{\angn_{\gamma}(v)\right\}_{v\in P_q(a)}$ are distributed in $[0, \pi]$ when $a\to\infty$. 
Therefore, we spend the next two sections describing this distribution in some detail.

\section{Small angles of Kloosterman sums}
\label{sec.small.angles}
 
 In this section, we work in the following setting.
 Let $\F$ be a finite field of odd characteristic $p$, and  $\psi$ be a nontrivial additive character on $\F$. 
 We assume that $\psi$ takes values in $\Q(\zeta_p)$ and we pick a complex embedding $\Q(\zeta_p)\into\C$. 
 
For any $\alpha\in \F^\times$ the Kloosterman sum $\Kl_\F(\psi;\alpha)$ is a real algebraic integer with $|\Kl_\F(\psi;\alpha)|\leq 2|\F|^{1/2}$ (see Proposition \ref{prop.kloo}\eqref{kloo.item1}, \eqref{kloo.item4} and \eqref{kloo.item5}). 
Thus there is a well-defined angle $\theta_\F(\psi,\alpha)\in[0,\pi]$ associated to the Kloosterman sum by 
\[\Kl_\F(\psi;\alpha) := 2|\F|^{1/2}\cdot \cos \theta_\F(\psi;\alpha).\]
Further, as noted in Proposition \ref{prop.kloo}\eqref{kloo.item7}, the angle $\theta_\F(\psi;\alpha)$  cannot be $0$ or $\pi$ since $\Kl_\F(\psi;\alpha)$  ``never attains the Weil bound''. 
In this section, we investigate how close  $\theta_\F(\psi;\alpha)$ can be to $0$ and $\pi$; we prove the following result, which may be of independent interest.

\begin{theo}\label{theo.minangle} 
There exists an effectively computable constant $c_p>0$ (depending at most on $p$) such that the following holds:
for any finite field $\F$ of characteristic $p$,  any nontrivial additive character $\psi$ on $\F$ and any $\alpha\in\F^\times$, one has
\begin{equation}\notag
 |\theta_\F(\psi;\alpha)| > |\F|^{-c_p}
\qquad \text{ and } \qquad
|\pi-\theta_\F(\psi;\alpha)| > |\F|^{-c_p}.
\end{equation}
Moreover $c_p=2(p-1)$ is a suitable value of the constant.
\end{theo}

Before we start the proof, we recall for convenience the following version of  Liouville's inequality:
\begin{theo}[Liouville's inequality]\label{theo.liouville}
Let $P\in\Z[X]$ be a polynomial of degree $N$.
For any algebraic number $z\in\Qbar$, let $D_z$ be its degree over $\Q$ and $h(z)$ denote its logarithmic absolute Weil height.

Either $P(z)= 0$ or
\begin{equation}\label{eq.liouville}
|P(z)| \geq \|P\|_1^{-D_z+1} \cdot \exp\left(-N\cdot D_z \cdot h(z)\right),
\end{equation}
in any complex embedding of $\Q(z)$, where $\|P\|_1$ is the sum of the absolute values of the coefficients of $P$.
\end{theo}
See the introduction of \cite{MiWa} and the proof of Lemma 5 in \emph{loc.cit.} for this version and its proof.

    \begin{proof}[of Theorem \ref{theo.minangle}]
    We let $\F, \psi$ and $\alpha$ be as in the statement of Theorem \ref{theo.minangle}; we choose an embedding $\Qbar\into\C$ which is compatible with our choice of $\Q(\zeta_p)\into\C$.  By Proposition \ref{prop.kloo}\eqref{kloo.item4}, we can write
    \[\Kl_\F(\psi;\alpha) = \kkappa_{\F}(\psi; \alpha) + \kkappa'_{\F}(\psi;\alpha),\]
    for two algebraic integers $\kkappa_{\F}(\psi; \alpha)$, $\kkappa'_{\F}(\psi;\alpha)$ of magnitude $|\F|^{1/2}$ whose product is $|\F|$. 
    In the given complex embedding, one of $\kkappa_{\F}(\psi; \alpha)$ and $\kkappa'_{\F}(\psi;\alpha)$ equals $|\F|^{1/2} \cdot \e^{i\ang_\F(\psi;\alpha)}$ by Proposition \ref{prop.kloo}\eqref{kloo.item5}. Without loss of generality, we assume that $\kkappa_\F(\psi;\alpha)=|\F|^{1/2} \cdot \e^{i\ang_\F(\psi;\alpha)}$.
    
    We let $z\in\bar{\Q}$ be 
    the ratio ${z:=\kkappa_\F(\psi;\alpha)/|\F|^{1/2}} =\e^{i\theta_\F(\psi;\alpha)}$
    and $L:=\Q(z)\subset\Qbar$. 
    We have 
    \begin{lemm}\label{lemm.minangle} 
    $z$ has degree $D_z\leq 2(p-1)$ and height $h(z)\leq \log \sqrt{|\F|}$. Moreover, $z\neq\pm1$.
    \end{lemm}
        \begin{proof} 
        By Proposition \ref{prop.kloo}\eqref{kloo.item1}, we have $\Kl_\F(\psi; \alpha)\in\Q(\zeta_p+\zeta_p^{-1})$. 
        Further,  Proposition \ref{prop.kloo}\eqref{kloo.item4} implies that 
        $\kkappa_{\F}(\psi; \alpha)$ and $\kkappa'_{\F}(\psi;\alpha)$ are the roots of $X^2 - \Kl_\F(\psi;\alpha) \cdot X +|\F|$; it is then clear that the degree of $\kkappa_{\F}(\psi; \alpha)$ over $\Q$ is $\leq 2 [\Q(\zeta_p+\zeta_p^{-1}):\Q] = p-1$.
        Thus we have
        \[D_z = [L:\Q] \leq [\Q\big(|\F|^{1/2}, \kkappa_{\F}(\psi; \alpha)\big):\Q] \leq 2[\Q(\zeta_p+\zeta_p^{-1}):\Q]\leq 2(p-1), \]
        as was to be shown. 
        Since $|\kkappa_\F(\psi; \alpha)|=|\F|^{1/2}$, 
        we infer that $|z|=1$ in any complex embedding of $L$,  therefore the archimedean places of $L$ do not contribute to $h(z)$.
        We deduce further that $z$ is a unit at all finite places of $L$ which are not above $p$; hence these places do not contribute to $h(z)$ either. 
        It remains to consider finite places $\mathfrak{P}$ of $L$ lying above $p$:
        since both $\kkappa_\F(\psi;\alpha)$ and $\kkappa'_\F(\psi;\alpha) = |\F|/\kkappa_\F(\psi;\alpha)$ are algebraic integers, the contributions of $\mathfrak{P}$ to $h(z)$ is 
        $\leq \frac{1}{2}\cdot\frac{\ord_{\mathfrak{P}}(|\F|)}{[L:\Q]}$. 
        By summing these, we conclude that $h(z)\leq \frac{1}{2}\log |\F|$. 
        
        Finally, Proposition \ref{prop.kloo}\eqref{kloo.item7} shows that $\theta_{\F}(\psi;\alpha)\notin\{ 0, \pi\}$, and 
        the last assertion easily follows.
        %
        \ProofEnd\end{proof}
    
    Upon applying Liouville's inequality \eqref{eq.liouville} to $z$ with $P=X\pm1$ and using Lemma \ref{lemm.minangle}, we obtain that
    \begin{align}
    \log|z \pm 1|& 
    \geq (-D_z+1)\cdot\log 2  - D_z \cdot h(z)
    \geq -D_z\cdot \left(\log 2  +\log \sqrt{|\F|}\right) \notag\\
    &> - D_z \cdot \log |\F|\geq -2(p-1)\cdot \log |\F|. \label{eq.minangle.interm}
    \end{align}
    Noting that the the graph of $t\mapsto |\sin t|$ lies below its tangents at $t=\pi/2$, one sees that $|\e^{i t} -1| \leq |t|$ for all $t\in[0,\pi]$.
    From the lower bound \eqref{eq.minangle.interm} on $|z-1|$, we deduce from this inequality that 
    \[ |\theta_\F(\psi;\alpha)| \geq |\e^{i\theta_\F(\psi;\alpha)}-1| =|z-1|> |\F|^{-c_p},\]
    where $c_p = 2(p-1)$. To get the lower bound on $|\theta_\F(\psi;\alpha) -\pi|$, we use the same inequality with $t$ replaced by $t'=\pi - t$ and the lower bound on $|z+1|$ in \eqref{eq.minangle.interm}.  This concludes the proof of Theorem \ref{theo.minangle}.
    \ProofEnd
    \end{proof}
    

Let us deduce two corollaries from Theorem \ref{theo.minangle}. The first one can be viewed as a slight improvement on the Weil bound on Kloosterman sums (\ie{}, an effective version of Proposition \ref{prop.kloo}\eqref{kloo.item7}):
\begin{coro} For any finite field $\F$ of characteristic $p\geq 3$, any nontrivial additive character $\psi$ on $\F$ and any $\alpha\in\F^\times$, we have
\[  |\Kl_\F(\psi;\alpha)| 
\leq 2|\F|^{1/2}\cdot\left(1-\frac{2}{\pi^2} \cdot |\F|^{-2c_p}\right),\]
where $c_p$ is the constant in Theorem \ref{theo.minangle}. 
\end{coro}

    \begin{proof} 
    By construction, we have $|\Kl_\F(\psi;\alpha)| = 2|\F|^{1/2}\cdot|\cos\theta_\F(\psi;\alpha)|$. Theorem \ref{theo.minangle}  implies that $\theta_\F(\psi; \alpha)$ lies in $[Q, \pi-Q]$ with $Q=(q^a)^{-c_p}$.
    The Corollary follows from the elementary observation that 
    \[\forall \theta\in[Q, \pi-Q], \quad 
    |\cos \theta| = \sqrt{1-\sin^2\theta}\leq 1-\frac{\sin^2 \theta}{2} \leq 1 - \frac{2\min\{\theta, \pi-\theta\}^2}{\pi^2}
    \leq 1- \frac{2Q^2}{\pi^2}. \eqno\square\]
    \end{proof}

The second corollary is more central to our study of the size of $L^\ast(E_{a, \gamma}, 1)$, \cf{} \S\ref{subsec.ST.limit}. 
  
\begin{coro}\label{coro.minangle}
Let $\F_q$ be a finite field of characteristic $p\geq 3$ equipped with a nontrivial additive character~$\psi_q$.  
For any $\gamma\in\F_q^\times$ and any integer $a\geq 1$, the angles $\angn_{\gamma}(v)$ for $v\in P_q(a)$ defined in \S\ref{subsec.anglesnorm} satisfy
\begin{equation}
{(q^a)}^{-c_p} \leq \angn_{\gamma}(v)\leq \pi - {(q^a)}^{-c_p},
\end{equation}
where $c_p$ is the constant 
in Theorem \ref{theo.minangle} above.
\end{coro}
 
    \begin{proof}
    For all $v\in P_q(a)$, the residue field $\F_v$ is a subfield of $\F_{q^a}$ and we may choose $\beta_v$ as in \S\ref{subsec.kloos.norm}. 
    Upon noting that $|\F_v|\leq q^a$,    the Corollary follows immediately from Theorem \ref{theo.minangle} applied to $\F=\F_{v}$, $\psi = \psi_q\circ\trace_{\F_v/\F_q}$ and $\alpha=\gamma\beta_v^2$. 
    \ProofEnd \end{proof}
 
\section[Equidistribution]{Distribution of the sums $\Kln_{\gamma}(v)$}
\label{sec.distrib}
    
In this section, we fix again a finite field $\F_q$ of odd characteristic $p$, an element $\gamma\in\F_q^\times$ and a nontrivial additive character $\psi_q$ on $\F_q$ with values in $\Q(\zeta_p)$. 
For any finite extension $\F/\F_q$, we continue denoting by $\psi_\F$ the composition $\psi_q\circ\trace_{\F/\F_q}$ of  $\psi_q$ with the trace $\trace_{\F/\F_q} : \F\to\F_q$.

Loosely speaking, we show that, asymptotically as $a\to\infty$, the numbers $q^{-d_v/2}\cdot\Kln_\gamma(v)$ with $v\in P_q(a)$ (see \S\ref{subsec.kloos.norm}) are distributed in $[-2, 2]$ as ``the traces of random matrices in $\SU(2, \C)$''. In order to make this statement more precise and to prove it, 
we begin by introducing the necessary notations and notions.

Choose a prime number $\ell\neq p$, an algebraic closure $\Qellbar$ of $\Q_\ell$, an embedding $\Qbar\into\Qellbar$, and 
a field isomorphism $\Qellbar\simeq \C$. 
 Through this isomorphism, we view $\psi_q$ as a $\Qellbar$-valued additive character on~$\F_q$.
%

We fix a separable closure $K\sep$ of $K$. 
The set of places 
$v\neq 0, \infty$ of $K$ can be identified with the set of closed points of the multiplicative group $\G_m=\P^1\smallsetminus\{0, \infty\}$ over $\F_q$. 
For a finite extension  $\F/\F_q$ and a point $\alpha\in\G_m(\F)$, 
we denote by $\Frob_{\F, \alpha}$ the geometric Frobenius of $\G_m$ at $\alpha$, which we view as a conjugacy class in the profinite group $\Gal(K\sep/K)$. 
For any closed point $v$ of $\G_m$, we choose $\beta_v\in v$ and we let $\Frob_v:=\Frob_{\F_v, \beta_v}$. 


\subsection{Angles of Kloosterman sums}
\label{subsec.angles.setting}

Let us start by redefining the angles $\ang_\F(\psi_\F;\alpha)$ from a representation-theoretic point of view.
The reader is referred to \cite[Chapter 3]{Katz_GKM} or \cite{Fisher} for more detailed presentations.

In Chapter 4 of \cite{Katz_GKM}, Katz has constructed 
a lisse $\Qellbar$-sheaf $\shKl^\circ$ on $\G_m$ whose Frobenius traces are Kloosterman sums ($\shKl^\circ$ is the so-called Kloosterman sheaf). Taking a suitable Tate twist, one obtains 
a lisse $\Qellbar$-sheaf $\shKl=\shKl^\circ(1/2)$  of rank $2$ on $\G_m$ 
which is pure of weight $0$. 

By definition, $\shKl$ ``is'' a continuous $2$-dimensional $\Qellbar$-representation $\kappa: \Gal(K\sep/K)\to   \GL(2, \Qellbar)$ 
  which is unramified outside $\{0, \infty\}$ and which satisfies the following.  
   For all places $v\neq 0, \infty$ of $K$,   the eigenvalues of $\kappa(\Frob_v)$ have magnitude\footnote{Here we view $\kappa(\Frob_v)\in\Qellbar$ as an element of $\C$ by means of the chosen isomorphism $\Qellbar\simeq\C$} $1$ (``pure of weight 0'') and the trace  of $\kappa(\Frob_v)$ is:
\[   \Trace\big(\kappa(\Frob_{v}) \big)
= |\F_v|^{-1/2}\cdot\Kl_{\F_v}(\psi_{\F_v};\beta_v),\]
where $\beta_v\in\G_m(\bar{\F_q})$ is a choice of element in the closed point of $\G_m$ corresponding to $v$ (as in \S\ref{subsec.kloos.norm}). 
Note that, even though $\Frob_v$ is only defined up to conjugation in $\Gal(K\sep/K)$, its $\Trace\big(\kappa(\Frob_{v}) \big)$ is well-defined.
 
Katz has shown that the image of $\Gal(K\sep/K)$ under $\kappa$ is contained in $\SL(2, \Qellbar)$ (in other words, the representation $\kappa$ has trivial determinant, see \cite[Chap. 11]{Katz_GKM}).
\emph{Via} the chosen isomorphism $\Qellbar\simeq\C$, 
we may view 
$\kappa(\Gal(K\sep/K))$ as a subgroup of $\SL(2, \C)$.
The special unitary group $K:=\SU(2, \C)$ is a maximal compact subgroup of $\SL(2, \C)$ and, since $\SL(2, \C)$ is semisimple, such a $K$ is uniquely determined up to conjugation.
For any place $v\neq 0, \infty$, let 
$\kappa(\Frob_v)\ss$ be the semisimplification of $\kappa(\Frob_v)$: the closure of the subgroup of  $\SL(2, \C)$ generated by all the $\kappa(\Frob_v)\ss$ is compact and thus, up to conjugation in $\SL(2, \C)$, 
lies in $K$.


We denote by $K^\natural$ the set of conjugacy classes of $K$ and we equip $K^\natural$ with 
the measure $\mu^\natural$ obtained as the direct image 
of the Haar measure on $K$ normalised to have total mass $1$. 
The trace of $M\in K$ (or of any element in its conjugacy class) is the 
sum of two conjugate complex number of magnitude $1$, so it is a real number in $[-2,2]$.
More precisely, a matrix $M\in K$ is conjugate (in $K$) to a diagonal matrix  $\mathrm{Diag}(\e^{i\theta_M}, \e^{-i\theta_M})$ for some unique $\theta_M\in[0,\pi]$ and $\Trace M = 2\cos\theta_M$.
Hence, the set $K^\natural$ endowed with~$\mu^\natural$ can be identified with the interval $[0, \pi]$ endowed with the Sato--Tate measure $\muST  := \frac{2}{\pi} \sin^2\theta\dd \theta$ (see \cite[Chap. 13]{Katz_GKM}). 
We identify any angle $\theta\in[0, \pi]$ with the conjugacy class of $\mathrm{Diag}(\e^{i\theta}, \e^{-i\theta}) \in K^\natural$, which we also denote by the same symbol $\theta$.
 
 We are now ready to (re)define angles of Kloosterman sums.
 For any finite extension $\F/\F_q$ and any $\alpha\in\G_m(\F)$, the semisimplification of $\kappa(\Frob_{\F, \alpha}) \in \SL(2,\C)$ is $\SL(2,\C)$-conjugate 
 to an element of $K$, and we can define $\theta_\F(\psi_\F; \alpha)\in K^\natural$ to be the conjugacy class (in $K$) of this element. 
  In the identification between~$K^\natural$ and $[0, \pi]$, this gives us a well-defined angle $\theta_\F(\psi_\F; \alpha)\in[0,\pi]$, see \cite[\S3.3]{Katz_GKM}. 
  
For any finite extension $\F/\F_q$ and any $\alpha\in\G_m(\F)$, we thus have
  \[2\cdot \cos\theta_\F(\psi_\F; \alpha)=  \Trace\big(\kappa(\Frob_{\F, \alpha})\ss \big)= \Trace\big(\kappa(\Frob_{\F, \alpha}) \big)
= |\F|^{-1/2}\cdot\Kl_{\F}(\psi_{\F};\alpha),\]
so that the new definition of $\cos\theta_\F(\psi_\F; \alpha)$ coincides with the one given at the beginning of section \ref{sec.small.angles}.

    \begin{defi} 
    Fix a finite field $\F_q$ equipped with a nontrivial additive character $\psi_q$ and $\gamma\in\F_q^\times$. 
    For any place $v\neq 0, \infty$ of $K$, let $\angn_\gamma(v)$ be the angle associated to the Kloosterman sum $\Kln_\gamma(v)=\Kl_{\F_v}(\psi_{\F_v}; \gamma\beta_v^2)$ by the construction above. 
    In other words, we put $\angn_\gamma(v) := \theta_{\F_v}(\psi_{\F_v};\gamma\beta_v^2)\in K^\natural$.
    \end{defi}

\begin{rema}\label{rema.rela.angles}
Let $\F$ be the finite extension of $\F_q$ with $[\F:\F_q]=a$. 
For an element $\alpha\in\G_m(\F)$, let $w$ be the closed point of $\G_m$ corresponding to $\alpha$ (\ie{}, $w$ is the $\Gal(\bar{\F_q}/\F_q)$-orbit of $\alpha$).
The residue field $\F_w$ is then a subfield of $\F$ and $\Frob_{\F, \alpha} = (\Frob_w)^{[\F:\F_w]}$ as conjugacy classes in $\Gal(K\sep/K)$. 
Therefore $\kappa(\Frob_{\F, \alpha}) = \kappa(\Frob_w)^{[\F:\F_w]}$ and
\[ \tfrac{a}{\deg w} \cdot \theta_{\F_w}(\psi_{\F_w};\alpha)  = [\F:\F_w] \cdot \theta_{\F_w}(\psi_{\F_w};\alpha) \equiv \theta_{\F}(\psi_\F;\alpha) \bmod{\pi}. \]
In particular, when $v$ is a closed point of $\G_m$ whose degree divides $a$ and when $\beta_v\in v$, by  our definition  $\angn_{\gamma}(v) = \ang_{\F_v}(\psi_{\F_v}; \gamma\beta_v^2)$, the relation above reads
\begin{equation}\label{eq.rela.angles}
\tfrac{a}{\deg v}\cdot\angn_{\gamma}(v) 
\equiv  \ang_{\F}(\psi_{\F}; \gamma\beta_v^2) \bmod{\pi}.
\end{equation}
\end{rema}

\subsection{Statement of results}
\label{subsec.setting}
  
Denote by $\muST $ the Sato--Tate measure $\frac{2}{\pi}\sin^2\theta\dd \theta$ on $[0,\pi]$. 
A sequence of Borel measures $\{\mu_i\}_{i\geq 1}$ on $[0,\pi]$ is said to converge weak-$\ast$ to $\muST $ if, for every continuous $\C$-valued function $f$ on $[0, \pi]$, the sequence of $\int_{[0, \pi]} f\dd\mu_i$ converges to $\int_{[0,\pi]}f\dd\muST $ as $i\to\infty$.
 
\paragraph{}
Our results concern two sequences of probability measures that we now introduce. 

\begin{defi}
We fix a finite field $\F_q$ of characteristic $p\geq 3$, a nontrivial additive character $\psi_q$ on $\F_q$ and~$\gamma\in\F_q^\times$.
For an integer $a\geq 1$, we again denote by $P_q(a)$ the set of closed points of $\G_m$ whose degrees divide $a$. For all integers $a\geq 1$, we define
\begin{equation}\notag{}
\nu{(\F_q, \psi_q, \gamma; a)}
:= \frac{1}{|P_q(a)|} \cdot \sum_{v\in P_q(a)} \delta{\{\angn_{\gamma}(v)\}},
\end{equation}
where $\delta{\{x\}}$ denotes the Dirac delta measure at $x\in [0,\pi]$.
For all finite extensions $\F/\F_q$, we also define
\begin{equation}\notag{}
\xi{(\F_q, \psi_q, \gamma; \F)}
:= \frac{1}{|\G_m(\F)|} \cdot \sum_{\beta\in\G_m(\F)} \delta{\left\{\theta_{\F}(\psi_{\F}; \gamma\beta^2)\right\}}. 
\end{equation}
In what follows, we abbreviate $\nu{(\F_q, \psi_q, \gamma; a)}$ by $\nu_a$ and   $\xi{(\F_q, \psi_q, \gamma; \F_{q^a})}$ by $\xi_a$.
\end{defi}

Clearly, both $\nu_a$ and $\xi_a$ are Borel measures on $[0,\pi]$ with total mass $1$. 
It follows from the discussion in~\S\ref{subsec.angles.setting} that we may view  $\nu_a$ and $\xi_a$ as measures on $K^\natural$; we use both points of view interchangeably.
Moreover, we note that $\xi_a$ is also given by\footnote{The measure $\xi_a$ is the measure denoted by $X_a$ in \cite[\S3.5]{Katz_GKM} applied to our situation.}  
\begin{equation}\label{eq.rela.xi}
\xi_a = \xi{(\F_q, \psi_q, \gamma; \F_{q^a})}
= \frac{1}{|\G_m(\F_{q^a})|} \cdot \sum_{v\in P_q(a)} \deg v\cdot\delta{\big\{\tfrac{a}{\deg v}\cdot\angn_{\gamma}(v)\big\}}, 
\end{equation}
where $\tfrac{a}{\deg v}\cdot\angn_{\gamma}(v)$ is to be understood modulo $\pi$ (see Remark \ref{rema.rela.angles}). 

\begin{rema}\label{rema.minangle.suppmeas}
In terms of the measure $\nu_a$, Corollary \ref{coro.minangle} can be reinterpreted as follows: given  $\F_q, \psi_q$ and $\gamma$ as above, for any $a\geq 1$ the support of the probability measure $\nu_a$ on $[0, \pi]$  is contained in $[(q^a)^{-c_p}, \pi  -  (q^a)^{-c_p}]$. 
\end{rema}

\paragraph{}
We can now state the two main results of this section. 
First we show that the angles $\{\angn_\gamma(v)\}_{v\in P_q(a)}$ are asymptotically equidistributed with respect to the Sato--Tate measure as $a\to\infty$. 
Namely, 

\begin{theo}\label{theo.ST} 
Assume we are given a datum $\F_q, \psi_q, \gamma$ as above. 
Then the sequences $\{\xi_a\}_{a\geq 1}$ and $\{\nu_a\}_{a\geq 1}$ of Borel probability measures both converge weak-$\ast$ to the Sato--Tate measure $\muST $ when $a\to\infty$.
\end{theo}

This statement concretely means that, for all continuous functions $f$ on $[0,\pi]$, we have
\begin{equation}\label{eq.ST.concrete}
\frac{1}{|P_q(a)|} \cdot \sum_{v\in P_q(a)} f\big(\angn_{\gamma}(v)\big)
=\int_{[0,\pi]} f \dd\nu_a \xrightarrow[a\to\infty]{} \int_{[0, \pi]} f\dd\muST  = \frac{2}{\pi} \int_0^\pi f(t)\sin^2(t)\dd t.
\end{equation}
It will be proven in Propositions \ref{prop.ST1} and \ref{prop.ST2} by a suitable adaptation of the arguments in \cite[Chap.3]{Katz_GKM} and \cite[\S2]{FuLiu}. 

In the course of proving Theorem \ref{theo.ST.ext}, we will need a more effective version of \eqref{eq.ST.concrete}: indeed, we require an estimate of the rate at which  $\int_{[0,\pi]} f \dd\nu_a$ converges to $\int_{[0, \pi]} f\dd\muST $, at least for a smaller class of functions~$f$. 
This is the object of the second result in this section: 

\begin{theo}\label{theo.ST.eff} 
Assume we are given a datum $\F_q, \psi_q, \gamma$ as above. For any continuously differentiable function $g$ on $[0,\pi]$, we have
\begin{equation}\label{eq.ST.effective}
\left|\int_{[0,\pi]} g \dd\nu_a - \int_{[0, \pi]} g\dd\muST  \right| \ll_q \frac{a^{1/2}}{q^{a/4}} \cdot \int_{0}^\pi |g'(t)|\dd t \qquad (\text{as } a\to\infty).
\end{equation}
\end{theo}

This will follow from the proof of Theorem \ref{theo.ST} coupled with tools from distribution theory (see \cite{Niederreiter_Klo}).

We remark that the constants in Theorems \ref{theo.ST} and  \ref{theo.ST.eff} depend at most on $q$ (and neither on the choice of $\psi_q$  nor on the value of $\gamma\in\F_q^\times$).

\subsection[Equidistribution of the angles]{Equidistribution of $\angn_\gamma(v)$'s}
\label{subsec.equidis}
  
In this subsection, we prove Theorem \ref{theo.ST} in two steps (Propositions \ref{prop.ST1} and \ref{prop.ST2}). 
Let us first make a reduction (see \cite[\S3.4, \S3.5]{Katz_GKM} for more details). 
We need to show that, for all complex-valued continuous functions $f$ on $[0,\pi]$, we have
\begin{equation}\notag{}
\int_{[0,\pi]} f \dd\xi_a, \ \ \int_{[0,\pi]} f \dd\nu_a \xrightarrow[a\to\infty]{} \int_{[0, \pi]} f\dd\muST.
\end{equation}
Since $K$ is compact and since $[0,\pi]$ can be identified with $K^\natural$, there is a natural correspondence between the space of continuous functions on $[0,\pi]$ and the space $\mathscr{C}^0_{cent}(K)$ of continuous central functions on ${K=\SU(2, \C)}$.
When $\mathscr{C}^0_{cent}(K)$ is endowed with the topology of the supremum norm, the Peter-Weyl theorem asserts that the vector subspace generated by characters of irreducible finite-dimensional representations of $K$ is dense in $\mathscr{C}^0_{cent}(K)$. 
By density, Theorem \ref{theo.ST} will follow if we can show that, for all irreducible finite-dimensional representations $\Lambda$ of $K$,
\begin{equation}\notag{}
\int_{K^\natural} \Trace\Lambda \dd\xi_a, \ \ \int_{K^\natural} \Trace\Lambda \dd\nu_a \xrightarrow[a\to\infty]{} \int_{K^\natural} \Trace\Lambda\dd\mu^\natural.
\end{equation}
If $\Lambda_0$ is the trivial representation of $K$, $\Trace\Lambda_0$ is the trivial character $\trivcar$ on $K$ and the above limits trivially hold because the measures $\nu_a$, $\xi_a$ and $\mu^\natural$ on $K^\natural$ all have total mass $1$. 
By orthogonality of characters, the integral $\int_{K^\natural}\Trace\Lambda\dd\mu^\natural$ on the right-hand side vanishes when $\Lambda$ is a nontrivial irreducible finite-dimensional representation of $K$.
Hence,  the proof of Theorem \ref{theo.ST} reduces to that of the following statement, which is an analogue to Weyl's criterion for uniform distribution (see \cite[Chap. 4, \S1]{KuNi}) in the Sato--Tate context: for any nontrivial irreducible finite-dimensional representation $\Lambda$ of $\SU(2, \C)$, one has
\begin{equation}\notag{}
\int_{K^\natural} \Trace\Lambda \dd\xi_a, \ \ \int_{K^\natural} \Trace\Lambda \dd\nu_a \xrightarrow[a\to\infty]{} 0.
\end{equation}  
We actually prove slightly more precise estimates.
 
\begin{prop}\label{prop.ST1}
Fix $\F_q, \psi_q$ and $\gamma\in\F_q^\times$ as in \S\ref{subsec.setting}.
Let $\Lambda$ be a nontrivial irreducible representation of~$K=\SU(2, \C)$. For all $a\geq 1$, one has
\begin{equation}\label{eq.ST.meas1}
\left| \int_{K^\natural} \Trace\Lambda  \dd\xi_{a}\right|
\ll_q   \frac{\dim\Lambda}{q^{a/2}}.
\end{equation}
\end{prop}

	\begin{proof}  
	For any finite extension $\F/\F_q$, notice that 
	\begin{align*}
	\sum_{\beta\in\F^\times} \Trace\Lambda\big(\theta_\F(\psi_\F; \gamma\beta^2)\big)
	&= \sum_{\alpha'\in\F^\times} \left(1+\lambda_\F(\alpha')\right)\cdot\Trace\Lambda\big(\theta_\F(\psi_\F; \gamma\alpha')\big) \\
	&= \sum_{\alpha\in\F^\times} \Trace\Lambda\big(\theta_\F(\psi_\F;\alpha)\big) 
	+ \lambda_\F(\gamma^{-1}) \cdot \sum_{\alpha\in\F^\times} \lambda_\F(\alpha)\cdot\Trace\Lambda\big(\theta_\F(\psi_\F;\alpha)\big),
	\end{align*}
	where $\lambda_\F$ denotes the unique nontrivial character of order $2$ on $\F^\times$. 
	Therefore, one has
	\begin{equation}\label{eq.ST1.interm}
	\left|\sum_{\beta\in\F^\times} \Trace\Lambda\big(\theta_\F(\psi_\F; \gamma\beta^2)\big)\right|
	\leq \left|\sum_{\alpha\in\F^\times} \Trace\Lambda\big(\theta_\F(\psi_\F;\alpha)\big) \right| + \left|\sum_{\alpha\in\F^\times} \lambda_\F(\alpha)\cdot\Trace\Lambda\big(\theta_\F(\psi_\F;\alpha)\big)\right|.
	\end{equation}
	For any multiplicative character $\chi$ on $\F_q^\times$ and any finite extension $\F/\F_q$, we denote by $\chi_\F:=\chi\circ\norm_{\F/\F_q}$ the character on $\F^\times$ ``lifted'' by the norm $\norm_{\F/\F_q}:\F\to\F_q$.  
	The crucial input is a result of Fu and Liu (see Lemma 4 in \cite{FuLiu}) who have proved that, for every multiplicative character $\chi$ on $\F_q$, one has
	\begin{equation}\notag{}
	\left|\sum_{\alpha\in\G_m(\F)}\chi_\F(\alpha)\cdot \Trace\Lambda\big(\theta_\F(\psi_\F;\alpha)\big)\right| \leq \frac{\dim\Lambda}{2}\cdot |\F|^{1/2}.
	\end{equation}
	Applying this inequality successively to both multiplicative characters on $\F_q$ whose order divides $2$, we deduce from \eqref{eq.ST1.interm} that 
	\begin{equation}\notag{}
	\left|\sum_{\beta\in\F^\times} \Trace\Lambda\big(\theta_\F(\psi_\F; \gamma\beta^2)\big)\right| \leq (\dim\Lambda)\cdot |\F|^{1/2}.
	\end{equation}
	Therefore, for any finite extension $\F/\F_q$, we have proved that
	\[\left| \int_{K^\natural} \Trace\Lambda  \dd\xi{(\F_q, \psi_q, \gamma; \F)}\right| 
	= \left| \frac{1}{|\G_m(\F)|}\sum_{\beta\in\F^\times} \Trace\Lambda\big(\theta_\F(\psi_\F; \gamma\beta^2)\big)\right|
	\ll_q \dim \Lambda \cdot |\F|^{-1/2}.\]
	Specialising to $\F=\F_{q^a}$ yields the desired estimate since $\xi_a=\xi{(\F_q, \psi_q, \gamma; \F_{q^a})}$. 
	\ProofEnd\end{proof}

\begin{rema} 
Let us suggest an alternative way of proving  Proposition~\ref{prop.ST1}.
Denote again by $\shKl$ the Kloosterman sheaf (suitably twisted to be pure of weight $0$) whose existence was proved by Katz. 
For a given $\gamma\in\F_q^\times$, consider the morphism $f:\G_m\to\G_m$ given by $\beta\mapsto\gamma\beta^2$, and put $\shG_\gamma:=f^\ast\shKl$. 
Then $\shG_\gamma$ is again a lisse $\Qellbar$-sheaf of rank $2$ on $\G_m$ which is pure of weight $0$. 
In the case where $\gamma=1$, Remark 1 in \cite{FuLiu} sketches a proof that $\shG_\gamma$ satisfies the assumptions of \cite[\S3.1-\S3.3]{Katz_GKM}. 
One could therefore prove Proposition \ref{prop.ST1} 
by making use of \cite[\S3.6]{Katz_GKM}. 

This argument should carry over to the case of an arbitrary $\gamma\neq 0$.
\end{rema}


To complete the proof of Theorem \ref{theo.ST}, it remains to show that $\{\nu_a\}_{a\geq1}$ also converges (weak-$\ast$) to $\muST $.

\begin{prop}\label{prop.ST2}
Fix a datum $\F_q, \psi_q$ and $\gamma\in\F_q^\times$ as in \S\ref{subsec.setting}.
Let $\Lambda$ be a nontrivial irreducible representation of~$\SU(2, \C)$. For all $a\geq 1$, one has
\begin{equation}\label{eq.ST.meas2}
\left| \int_{K^\natural} \Trace\Lambda  \dd\nu_{a}\right|
\ll_q \dim\Lambda \cdot \frac{a}{q^{a/2}}.
\end{equation}
\end{prop}

	\begin{proof}
	Consider the measure $W_a := |\G_m(\F_{q^a})|\cdot \xi_a - a|P_q(a)|\cdot \nu_a$ on $K^\natural$ (or $[0,\pi]$); in other words, by \eqref{eq.rela.xi},
	\begin{align*}
	W_a  &= \sum_{ \substack{ v\in P_q(a) \\ \deg v <a}} \left( \deg v \cdot \delta{\big\{ \tfrac{a}{\deg v}  \angn_{\gamma}(v)\big\}} 
	- a \cdot \delta{\{\angn_{\gamma}(v)\}}\right)
	= \sum_{ \substack{  b \mid a \\ b <a}} \sum_{\substack{ v \text{ s.t.} \\\deg v= b}} \left( b \cdot \delta{\left\{\tfrac{a}{b} \angn_{\gamma}(v)\right\}} 
	- a \cdot \delta{\{\angn_{\gamma}(v)\}}\right).
	\end{align*}
	As is clear from the second expression, $W_a$ is a sum of $\sum_{ \substack{  b \mid a \\ b <a}} \pi_q(b)\cdot (b+a) $
	Dirac $\delta$ measures supported at points of $K^\natural$. 
	For any $z\in K$, the eigenvalues of $\Lambda(z)$ all have magnitude $1$, therefore we have ${\left|\Trace\Lambda(z)\right|\leq \dim\Lambda}$.
	Hence we find that
	\begin{equation}\label{eq.change.meas.inter}
	\left|\int_{K^\natural}\Trace\Lambda \dd W_a\right|
	\leq \left( \sum_{ \substack{  b \mid a \\ b <a}} \pi_q(b)\cdot (b+a) \right) \cdot \dim\Lambda
	\ll_q \dim \Lambda \cdot a \cdot q^{a/2}
	\end{equation}
	Indeed, straightforward estimates using Lemma \ref{lemm.estimates}\eqref{estim.item1} show that
	\[\sum_{ \substack{  b \mid a \\ b <a}} \pi_q(b)\cdot (b+a)
	\leq 2 a\cdot \sum_{ \substack{  b \mid a \\ b <a}} \pi_q(b)
	\ll_q a  \sum_{1\leq b\leq a/2} q^b
	\ll_q a \cdot q^{a/2}.
	\]

	Notice that 
	\[\nu_a = \frac{|\G_m(\F_{q^a})|}{a|P_q(a)|} \cdot \xi_a 
	+ \frac{1}{ a|P_q(a)|}\cdot W_a,\]
	where $a|P_q(a)|\gg_q q^a$ and 
	${|\G_m(\F_{q^a})|} \ll_q {a|P_q(a)|}$, again by Lemma \ref{lemm.estimates}. 
	Therefore, combining  \eqref{eq.ST.meas1}  in the previous Proposition 
	and inequality \eqref{eq.change.meas.inter}, we deduce that
	\begin{align*}
	\left|\int_{K^\natural}\Trace\Lambda  \dd\nu_a\right|
	&\leq \frac{|\G_m(\F_{q^a})|}{a|P_q(a)|} \cdot \left|\int_{K^\natural}\Trace\Lambda  \dd\xi_a\right| 
	+\frac{1}{ a|P_q(a)|} \left|\int_{K^\natural}\Trace\Lambda  \dd W_a\right| \\
	&\ll_q \dim \Lambda \cdot \left( q^{-a/2} + \frac{a \cdot q^{a/2}}{q^a} \right) 
	\ll_q \dim\Lambda \cdot \frac{a}{q^{a/2}}.
	\end{align*}
	This concludes the proof of the Proposition  and, by the discussion at the beginning of this subsection, that of Theorem \ref{theo.ST}.
	\ProofEnd\end{proof}

\subsection{Effectivity of the equidistribution}
\label{subsec.eff.ST}

The nontrivial irreducible representations 
of $K=\SU(2, \C)$ are exactly the symmetric powers $\symm^n(\mathrm{std})$ 
of the standard representation $\mathrm{std}:\SU(2, \C)\into \GL(2, \C)$. 
Moreover, if $\Lambda_n = \symm^n(\mathrm{std})$ for some~${n\geq 1}$, then $\Lambda_n$ has dimension $n+1$ and the trace function $\Trace\Lambda_n:K^\natural \to \R$ corresponds to the map%
\footnote{so that $\Trace\Lambda_n(\theta)=U_n(\cos\theta)$, where $U_n$ is the $n$-th Chebyshev polynomial of the second kind.} 
${\theta\mapsto  \sin\big((n+1)\theta\big)/\sin \theta}$ in the identification of $K^\natural$ with $[0, \pi]$.

It is convenient to denote by $\mom_n(a)$, the ``$n$-th moment'' of $\{\angn_\gamma(v)\}_{v\in P_q(a)}$ \ie{},
\begin{equation}\label{eq.defi.moment}
\mom_n(a):= \int_{K^\natural} \Trace \Lambda_n  \dd\nu_a 
= \frac{1}{|P_q(a)|} \cdot\sum_{v\in P_q(a)} \frac{\sin\big((n+1)\angn_{\gamma}(v)\big)}{\sin \angn_{\gamma}(v)}.
\end{equation}
With this notation, the result of Proposition \ref{prop.ST2} can be rewritten as follows. 
Given a datum $\F_q, \psi_q,\gamma\in\F_q^\times$ as in \S\ref{subsec.setting} and an integer $n\geq 1$, one has 
\begin{equation}\label{eq.coro.mom}
\forall a\geq 1, \qquad \left| \mom_n(a)\right|
\ll_q (n+1)\cdot {a}{q^{-a/2}}.
\end{equation} 

\paragraph{}
To measure ``how far'' from being perfectly equidistributed with respect to the Sato--Tate distribution are  the angles $\angn_{\gamma}(v)$,
 it is customary to introduce the \emph{star discrepancy} $\dis_{q, \gamma}(a)$:
\begin{equation}\notag
\dis_{q, \gamma}(a) :=  
\sup_{x}\left| \int_{[0, x)} \!\!\!\dd\nu_a - \int_{[0, x)}\!\!\!\dd\muST \right|
= \sup_{x}\left| \frac{\left| \{ v\in P_q(a) : \angn_{\gamma}(v)\in[0,x) \}\right|}{|P_q(a)|} - \int_{[0, x)} \!\!\!\dd\muST \right|,
\end{equation}
where the supremums are taken over $x\in[0,\pi]$. 
This definition is the direct analogue of the  star discrepancy for the uniform measure (see \cite[Chap.2, \S1]{KuNi}) in the context of $\muST $.
The interest of finding good upper bounds on $\dis_{q, \gamma}(a)$ is exemplified by the following result, which is similar to Koksma's inequality (see Theorem 5.1 in \cite[Chap.2]{KuNi}).

\begin{theo}[Niederreiter]\label{theo.koksma}
For any function  $g:[0,\pi]\to\R$ of total bounded variation $\totvar{g}$, one has
\begin{equation}\label{eq.koksma}
\left| \int_{[0,\pi]} g\dd\nu_a - \int_{[0,\pi]} g\dd\muST \right| \leq \dis_{q, \gamma}(a) \cdot \totvar{g}.
\end{equation}
\end{theo}

This statement is essentially Corollary 2 in \cite{Niederreiter_Klo}, the proof of which is based on an adaptation to the Sato--Tate context of the proof of Koksma's inequality for the uniform measure (see \cite[p. 143]{KuNi}).

Note that, for a continuously differentiable function $g$, one has $\totvar{g} = \int_0^\pi |g'(t)|\dd t$. 
Therefore, Theorem~\ref{theo.ST.eff} follows directly from Theorem \ref{theo.koksma} and the following: 

\begin{prop} 
The star discrepancy of $\{\angn_{\gamma}(v)\}_{v\in P_q(a)}$ is bounded by
\begin{equation}\notag{} 
\dis_{q, \gamma}(a)
\ll_q \frac{a^{1/2}}{q^{a/4}}.
\end{equation}
\end{prop}

	\begin{proof} 
	Niederreiter (see \cite{Niederreiter_Klo}) has proved a variant of the ``Erdös--Tur\'{a}n inequality'' in the Sato--Tate context.
	Just as the Erdös--Tur\'{a}n theorem  (see \cite[Chap.2, Thm.2.5]{KuNi}), his result gives an upper bound on $\dis_{q, \gamma}(a)$ in terms of ``exponential sums'', here the moments $\mom_n(a)$ defined in  \eqref{eq.defi.moment}. 
	Let us state Lemma~3 in \cite{Niederreiter_Klo} as follows\footnote{Niederreiter works on the interval $[-1, 1]$ endowed with the direct image of $\mu_{ST}$ under $t\mapsto\cos t$ (the ``semi-circle measure''); the translation to our setting is straightforward. Note that \cite{Niederreiter_Klo}   actually gives explicit constants in \eqref{eq.ET.ST}.}: 
	for any odd positive integer $N$, we have
	\begin{equation}\label{eq.ET.ST}
	\dis_{q, \gamma}(a) \ll \frac{1}{N} + \sum_{n=1}^{2N-1} \frac{n+1}{n(n+2)} \cdot |\mom_n(a)|,
	\end{equation}
	As was noted in \eqref{eq.coro.mom}, Proposition \ref{prop.ST2} reads: ${| \mom_n(a)|\ll_q (n+1)\cdot a q^{-a/2}}$.
	Also remark that
	\[ \sum_{n=1}^{2N-1} \frac{(n+1)^2}{n(n+2)} 
	\leq 2N - 1 +\sum_{n=1}^{\infty}\frac{1}{n^2+2n}  = 2N-1+3/4 \ll N.\]
	Hence, for all odd $N\geq 1$, \eqref{eq.ET.ST} leads to $\dis_{q, \gamma}(a) \ll_q  N^{-1} + a q^{-a/2} \cdot N$.
	Choosing $N$ to be the largest odd integer smaller than $(a^{-1}q^{a/2})^{1/2}$, we have $N^{-1}\ll a^{1/2}q^{-a/4}$ and we obtain the desired bound. 
	\ProofEnd\end{proof}

\section{Bounds on the special value}
\label{sec.bnd.spval} 

  By definition \eqref{defi.spval}, the special value $L^\ast(E_{a, \gamma}, 1)$ is the value at $T=q^{-1}$ of a polynomial with integral coefficients of degree $\leq b(E_{a, \gamma})$. Therefore, $L^\ast(E_{a, \gamma}, 1)$  is of the form 
  $n/q^{b(E_{a, \gamma})}$ for some integer $n\geq 1$, and we deduce the following ``trivial'' lower bound on $L^\ast(E_{a, \gamma}, 1)$:
  \begin{equation}\label{eq.triv.lwrbnd}
\frac{\log L^\ast(E_{a, \gamma}, 1)}{\log\left( q^{b(E_{a, \gamma})}\right)}= \frac{\log n}{\log\left( q^{b(E_{a, \gamma})}\right)} - 1 \geq - 1.
  \end{equation}

On the other hand, using techniques from classical complex analysis, Hindry and Pacheco 
show 
the following upper bound on $L^\ast(E_{a, \gamma}, 1)$  (see Theorem 7.5 in \cite{HP15}): 
\begin{equation*}
\frac{\log L^\ast(E_{a, \gamma}, 1)}{\log\left( q^{b(E_{a, \gamma})}\right)}\ll_q \frac{\log b(E_{a, \gamma})}{b(E_{a, \gamma})}.  
  \end{equation*}
 In this section, we prove the main theorem of this article (Theorem \ref{itheo.main} in the introduction), which provides a significant improvement on \eqref{eq.triv.lwrbnd}:
 
\begin{theo}\label{theo.bnd.spval}
 Let $\F_q$ be a finite field of odd characteristic $p$ and $K=\F_q(t)$.
There exist positive constants $C_1, C_2$, which depend at most on $q$ and $p$ such that the following holds. 
For all $\gamma\in\F_q^\times$ and all integers $a\geq 1$, the special value $L^\ast(E_{a, \gamma}, 1)$ satisfies: 
 \begin{equation}\label{eq.bnd.spval}
 - C_1\cdot \frac{1}{ \log b(E_{a, \gamma})} 
 \leq \frac{\log L^\ast(E_{a, \gamma}, 1)}{ \log \left(q^{b(E_{a, \gamma})}\right)}
 \leq C_2\cdot\frac{\log b(E_{a, \gamma})}{ b(E_{a, \gamma})}
 \quad (\text{as }a\to\infty).
 \end{equation}
 \end{theo} 

We will prove this Theorem in \S\ref{subsec.proof.bnd.spval}. 
The upper bound in \eqref{eq.bnd.spval} does not radically improve on the upper bound of Hindry and Pacheco (\emph{loc. cit.}) 
but, since our proof is rather short and elementary, 
we decided to include it here for completeness.
Our proof of the lower bound in \eqref{eq.bnd.spval}, on the other hand, is much more involved: the crucial step is the computation of a  ``Sato--Tate limit'', using the results of sections \ref{sec.small.angles} and \ref{sec.distrib} (see the next subsection). 
For later use (in section \ref{sec.BS}), we note that Theorem \ref{theo.bnd.spval} implies that
\begin{equation}\notag
\left|\log L^\ast(E_{a, \gamma}, 1)\right| = o\big(b(E_{a, \gamma})\big)  
\qquad (\text{as } a\to\infty).
\end{equation}

\subsection{Evaluation of a Sato--Tate limit}
\label{subsec.ST.limit}
 
In this subsection, we show the following result, which is the crucial input in our proof of the lower bound in Theorem \ref{theo.bnd.spval}.
For any integer $a\geq 1$, we again denote by $\nu_a=\nu(\F_q, \psi_q, \gamma;a)$ the probability measure on~$[0,\pi]$ introduced in \S\ref{subsec.setting}.

\begin{theo}\label{theo.ST.ext} 
Let $\F_q$ be a finite field equipped with a nontrivial additive character $\psi_q$, and $\gamma\in\F_q^\times$.
Then 
\begin{equation}\label{eq.ST.ext}
\int_{[0,\pi]}\log\sin^2 \dd\nu_a  \xrightarrow[a\to\infty]{} \int_{[0,\pi]}{\log\sin^2}\dd\muST. 
\end{equation}
\end{theo}
More concretely, this statement means that
\[\frac{1}{|P_q(a)|} \sum_{v\in P_q(a)} \log\left(\sin^2 \angn_{\gamma}(v) \right)
\xrightarrow[a\to\infty]{}
\frac{2}{\pi} \int_0^\pi \log(\sin^2t )\cdot \sin^2 t  \dd t = \log\frac{e}{4},\]
the evaluation of the integral on the right-hand side being a routine exercise in calculus.

	\begin{proof} 
	For conciseness, we denote by $w:[0,\pi]\to\R$ the function given by $w(t) := -\log(\sin^2 t)$ if $t\neq 0, \pi$ and $w(0)=w(\pi):=0$.
	Choose a  nondecreasing continuously differentiable function  $\phi_0:[0,1]\to\R$ such that $\phi_0 \equiv 0$ on $[0,1/3]$ and $\phi_0\equiv 1$ on $[2/3, 1]$.
	For all $\epsilon\in(0, 1)$, we define a function $\phi_\epsilon:[0,\pi]\to\R$ by
	\[ \phi_\epsilon(t) =
	\begin{cases}
	\phi_0(t/\epsilon) & \text{ if } t\in[0,\epsilon], \\
	1& \text{ if } t\in[\epsilon, \pi-\epsilon], \\
	\phi_0((\pi-t)/\epsilon) & \text{ if } t\in[\pi-\epsilon, \pi],
	\end{cases} 
	\]
	and we let $w_\epsilon := w\cdot \phi_\epsilon$.  By construction, $w_\epsilon$ is a continuously differentiable function on $[0,\pi]$ such that $w_\epsilon\leq w$ on $[0,\pi]$, $w\equiv w_\epsilon$ on $[\epsilon, \pi-\epsilon]$, and $w_\epsilon \equiv 0$ on $[0, \epsilon/3]\cup [\pi-\epsilon/3, \pi]$. 
	Furthermore, we have the following analytic estimates:

	\begin{lemm}\label{lemm.anal.estim}
	 Notations being as above, for all $\epsilon\in (0, 1)$, we have
	\begin{multicols}{2}
	\begin{enumerate}[(i)]
	\item\label{anal.item.1} $\displaystyle\int_0^\pi |w'_\epsilon(t)|\dd t \ll |\log\epsilon|$,
	\item\label{anal.item.2} $\displaystyle\int_0^\pi (w(t)-w_\epsilon(t))\cdot\sin^2(t) \cdot \dd t \ll \epsilon |\log\epsilon|.$
	\end{enumerate}
	\end{multicols}
	The constants depend only on the choice of $\phi_0$. 
	\end{lemm}

	We postpone the proof of this Lemma until the end of the subsection, and we now prove that 
	$\int_{[0, \pi]} w \dd\nu_a$ converges to  $\int_{[0, \pi]} w\dd \muST $ when $a\to\infty$. 
	For any $\epsilon\in (0,1)$, note that 
	\begin{align*}
	\left|\int  w \dd \nu_a - \int  w \dd \muST  \right|
	&\leq \underbrace{\int |w-w_\epsilon| \dd \nu_a}_{:=T_1}
	+  \underbrace{\left|\int  w_\epsilon \dd \nu_a  - \int w_\epsilon \dd \muST  \right|}_{:=T_2}
	+  \underbrace{ \int | w_\epsilon - w| \dd \muST }_{:=T_3}.
	\end{align*}
	Let us bound each of these three terms using the results in sections \ref{sec.small.angles} and \ref{sec.distrib}. 

	For $\epsilon>0$ sufficiently small, the first term $T_1$ vanishes. 
	Indeed, $w\equiv w_\epsilon$ on $[\epsilon, \pi-\epsilon]$ and, as we have proved,  the support of $\nu_a$ is contained in $[(q^a)^{-c_p}, \pi - (q^a)^{-c_p}]$ (see Remark \ref{rema.minangle.suppmeas}). 
	Therefore, for any $\epsilon < {(q^a)}^{-c_p} $, we have $T_1=0$.

	The function $w_\epsilon$ being continuously differentiable on $[0,\pi]$, we can use our effective equidistribution result (Theorem~\ref{theo.ST.eff}) to control the second term $T_2$. 
	Precisely, Theorem \ref{theo.ST.eff} yields
	\begin{equation*}
	T_2 = \left|\int  w_\epsilon \dd \nu_a  - \int w_\epsilon \dd \muST  \right|
	\ll_q \frac{a^{1/2}}{q^{a/4}} \cdot \int_{0}^\pi |w'_\epsilon(t)|\dd t 
	\ll_q \frac{a^{1/2}}{q^{a/4}}\cdot |\log\epsilon|, 
	\end{equation*}
	where the rightmost upper bound is Lemma \ref{lemm.anal.estim}\eqref{anal.item.1}.
	Finally, Lemma \ref{lemm.anal.estim}\eqref{anal.item.2} proves that $T_3 \ll \epsilon |\log\epsilon|$. 

	In summary, 
	for all $\epsilon>0$ such that $\epsilon < {(q^a)}^{-c_p}$, we have
	\[ \left|\int_{[0,\pi]}  w \dd \nu_a - \int_{[0,\pi]}  w \dd \muST  \right| 
	\ll_q \frac{a^{1/2}}{q^{a/4}}\cdot |\log \epsilon| + \epsilon|\log\epsilon|. \]
	Upon choosing $\epsilon$ of the form $\epsilon= (q^{a})^{ -\gamma}$ for some $\gamma>\max\{c_p, 4^{-1}\}$ (whose choice can be made to depend at most on $p$), 
	we conclude that
	\[\left|\int_{[0,\pi]}  w \dd \nu_a - \int_{[0,\pi]}  w \dd \muST  \right| 
	\ll_q \gamma a  \cdot \frac{   a^{1/2}}{q^{a/4}} + \frac{\gamma}{q^{a\cdot\gamma}}
	\ll_{q, p} \frac{a^{3/2}}{q^{a/4}},
	\]
	which proves Theorem \ref{theo.ST.ext}. 
	We even obtain a more quantitative version of \eqref{eq.ST.ext}: 
	\[\frac{1}{|P_q(a)|} \sum_{v\in P_q(a)} \log\left(\sin^2 \angn_{\gamma}(v)\right)
	= \log({\e}/{4}) + O\left(\frac{a^{3/2}}{q^{a/4}}\right) \quad (\text{as } a\to\infty), \]
	where the implicit constant depends at most on $q, p$ and the choice of the auxiliary function $\phi_0$.
	\ProofEnd \end{proof}

	\begin{proof}[of Lemma \ref{lemm.anal.estim}]
	Since both $w$ and $w_\epsilon$ are symmetric around $\pi/2$, it is sufficient to prove \eqref{anal.item.1} and \eqref{anal.item.2} where the integrals are replaced by integrals over $[0, \pi/2]$. 
	We also note that, for all $t\in(0, \pi/2]$, one has ${0\leq w(t)\leq 2\log\frac{\pi}{2t} }$. This follows from the classical estimate:  $\sin t\geq \frac{2t}{\pi}$ for $t\in [0, \pi/2]$.
	 
	To prove \eqref{anal.item.1}, we study separately the integrals over $(0,\epsilon)$ and over $[\epsilon, \pi/2]$.
	Since $w$ and $w_\epsilon$ coincide on $[\epsilon, \pi/2]$, we have 
	\[\forall t\in[\epsilon, \pi/2], \quad
	|w'_\epsilon(t)| = |w'(t)|  =2\left|\frac{\cos t}{\sin t}\right| \leq \frac{2}{|\sin t|} \leq \frac{\pi}{t},\]
	by the classical estimate mentioned above. 
	Hence we have
	$ \int_\epsilon^{\pi/2} |w'_\epsilon(t)|\dd t 
	\leq \pi \int_\epsilon^{\pi/2}t^{-1}\dd t
	\ll |\log\epsilon|$.
	On the interval $(0,\epsilon)$, we use the fact that $w_\epsilon = w\phi_\epsilon$ to deduce that
	\[\forall t\in(0,\epsilon), \quad |w'_\epsilon(t)| 
	\leq |w'(t)|\cdot \phi_\epsilon(t) + |w(t)| \cdot |\phi'_\epsilon(t)|
	\leq \frac{\pi}{t}\cdot \phi_\epsilon(t) + |w(t)| \cdot \frac{\|\phi'_0\|_\infty}{\epsilon}, 
	\]
	where $\|\phi_0'\|_\infty$ denotes the supnorm of $\phi'_0$ on $[0,\pi]$. 
	By the upper bound on $w(t)$ at the beginning of the proof and the fact that $\phi_\epsilon\equiv 0$ on $[0,\epsilon/3)$, we deduce that 
	\[\forall t\in(0,\epsilon), \quad |w'_\epsilon(t)| 
	\leq \frac{3\pi}{\epsilon} + \frac{2\|\phi'_0\|_\infty}{\epsilon} \cdot \log \frac{\pi}{2t},\]
	and from there, that 
	\[ \int_{0}^\epsilon |w'_\epsilon(t)|\dd t 
	\leq   3\pi  + \frac{2\|\phi'_0\|_\infty}{\epsilon} \cdot  \int_0^\epsilon \log \frac{\pi}{2t} \dd t
	 \ll 1 + \frac{\|\phi'_0\|_\infty}{\epsilon} \cdot \epsilon|\log \epsilon| 
	 \ll |\log \epsilon|.\]
	Summing the contributions of $\int_0^\epsilon$ and $\int_\epsilon^{\pi/2}$, we conclude that \eqref{anal.item.1} holds, with a constant depending only on~$\phi_0$.  We now show that \eqref{anal.item.2} holds: by the symmetry of $w-w_\epsilon$  and 
	since $w\equiv w_\epsilon$ on $[\epsilon, \pi/2]$, it suffices to prove that $\int_0^\epsilon |w(t)-w_\epsilon(t)|\sin^2t\dd t \ll \epsilon|\log\epsilon|$. By construction of $\phi_\epsilon$, we notice that 
	\[\int_0^\epsilon |w(t)-w_\epsilon(t)|\cdot\sin^2t \dd t 
	\leq\int_0^\epsilon  w(t) \cdot\sin^2t \dd t  
	\leq\int_0^\epsilon  w(t)   \dd t 
	\stackrel{(*)}{\leq} 2 \int_0^\epsilon \log \frac{\pi}{2t} \dd t
	\ll \epsilon |\log \epsilon|,
	\]
	where inequality $(*)$ follows from the upper bound on $w(t)$ at the beginning of the proof. 
	\ProofEnd \end{proof}

\subsection{Proof of Theorem \ref{theo.bnd.spval}}
\label{subsec.proof.bnd.spval}

 In \eqref{eq.expr.spval} 
 and \eqref{eq.expr.spval.alt}, we have proved that 
 \[ L^\ast(E_{a, \gamma}, 1) = \prod_{v\in P_q(a)}4\deg v \cdot \prod_{v\in P_q(a)} \left(1 -  \frac{\Kln_{\gamma}(v)^2}{4q^{\deg v}}\right) =   \prod_{v\in P_q(a)}4\deg v \cdot \prod_{v\in P_q(a)}  \sin^2\big(\angn_\gamma(v)\big).\]
Therefore, we have
\begin{equation}\label{eq.bnd.spval.inter1}
\frac{\log L^\ast(E_{a, \gamma}, 1) }{\log \big(q^{b(E_{a, \gamma})}\big)}
= \frac{1}{\log \big(q^{b(E_{a, \gamma})}\big)}\cdot\sum_{v\in P_q(a)} \log(4\deg v) + \frac{1}{\log \big(q^{b(E_{a, \gamma})}\big)} \sum_{v\in P_q(a)}  \log\big(\sin^2 \angn_{\gamma}(v)\big),
\end{equation}
and we estimate the two terms separately. First of all, let us prove that 
 \begin{equation}\label{eq.spval.upprbnd}
0\leq  \frac{1}{\log \big(q^{b(E_{a, \gamma})}\big)}\cdot\sum_{v\in P_q(a)} \log(4\deg v) 
 \ll_q \frac{\log a}{ a},  \qquad (\text{as }a\to\infty).
 \end{equation}
 The lower bound is clear, and the upper bound follows from Lemma \ref{lemm.estimates}\eqref{estim.item2}: indeed, we have $b(E_{a, \gamma}) \gg q^a$ (see \eqref{eq.degL}) and thus
  \[\sum_{v\in P_q(a)} \log(4\deg v) 
 \leq \log (4a) \cdot |P_q(a)|
 \ll_q  \frac{\log a \cdot q^a}{a},\]
The upper bound in \eqref{eq.spval.upprbnd} implies the upper bound in Theorem \ref{theo.bnd.spval} since the second term in \eqref{eq.bnd.spval.inter1}  is negative. To conclude the proof, it thus remains to prove that 
\[\left|\frac{1}{\log \big(q^{b(E_{a, \gamma})}\big)} \sum_{v\in P_q(a)}  \log\big(\sin^2 \angn_{\gamma}(v)\big)\right| \ll_{q, p} \frac{1}{\log b(E_{a, \gamma})}.\]
Let us write that
\[\left|\frac{1}{\log \big(q^{b(E_{a, \gamma})}\big)} \sum_{v\in P_q(a)}  \log\big(\sin^2 \angn_{\gamma}(v) \big)\right| 
\leq  \frac{|P_q(a)|}{\log \big(q^{b(E_{a, \gamma})}\big)}  \cdot \left|\frac{1}{|P_q(a)|} \sum_{v\in P_q(a)}  \log\big(\sin^2 \angn_{\gamma}(v) \big)\right|.
\]
By \eqref{eq.degL} again, we have $b(E_{a, \gamma})\gg q^a$, 
and Lemma \ref{lemm.estimates}\eqref{estim.item2} yields that $|P_q(a)|\ll_q q^a/a$. Moreover, by Theorem \ref{theo.ST.ext}, we have
\[\frac{1}{|P_q(a)|} \sum_{v\in P_q(a)}  \log\big(\sin^2 \angn_{\gamma}(v) \big) = \int_{[0,\pi]} \log\sin^2\dd\nu_a = \log(\e/4) + o(1)\qquad (\text{as } a\to\infty). \]
Putting these together, we obtain that 
\[\left|\frac{1}{\log \big(q^{b(E_{a, \gamma})}\big)} \sum_{v\in P_q(a)}  \log\big(\sin^2 \angn_{\gamma}(v) \big)\right| 
\ll_{q, p} \frac{1}{a} \ll_{q, p} \frac{1}{\log b(E_{a, \gamma})}.\]
This concludes the proof of Theorem \ref{theo.bnd.spval}.
\ProofEnd    
 
\section{Application to an analogue of the Brauer--Siegel theorem}
\label{sec.BS}
 
 In this section, we deduce from Theorem \ref{theo.bnd.spval} and from the BSD conjecture that the following theorem holds (stated as Theorem \ref{itheo.BS} in the introduction).
 
\begin{theo}\label{theo.BS}
Let $\F_q$ be a finite field of odd characteristic and $K:=\F_q(t)$.
For all $\gamma\in\F_q^\times$ and all integers $a\geq 1$, consider the elliptic curve $E_{a, \gamma}/K$ as above. 
Then the Tate--Shafarevich group $\sha(E_{a, \gamma})$ is finite and, as $a\to\infty$,
\begin{equation}\label{eq.BS}
\log\big(|\sha(E_{a, \gamma})|\cdot\Reg(E_{a, \gamma})\big) \sim \log H(E_{a, \gamma}). 
\end{equation}
\end{theo}

Alternatively, \eqref{eq.BS} can be rewritten under the form
\begin{equation}\label{eq.BS.alternative}
\forall \epsilon>0, \quad
 H(E_{a, \gamma})^{1-\epsilon} \ll_{q, \epsilon}
 |\sha(E_{a, \gamma})|\cdot\Reg(E_{a, \gamma}) 
 \ll_{q, \epsilon}  H(E_{a, \gamma})^{1+\epsilon} \qquad(\text{as }a\to\infty).
\end{equation}
The upper bound in \eqref{eq.BS.alternative} was essentially conjectured by Lang for elliptic curves over $\Q$ with finite Tate--Shafarevich groups\footnote{Note though that Lang uses a different normalisation of the height: his (naïve) height has an exponent $1/12$ instead of our exponent~$1$.} (see \cite[Conj.1]{Lang_conjecturesEC}). 
Theorem \ref{theo.BS} thus provides 
an unconditional example where this conjecture holds for elliptic curves over $K=\F_q(t)$. 
The lower bound in \eqref{eq.BS.alternative} further proves that the exponent $1$ of the height is optimal, in the sense that $1$ cannot be replaced by any smaller number. 
 
 
 \paragraph{}
One may also view Theorem \ref{theo.BS} as an analogue of the Brauer--Siegel theorem for the elliptic curves $E_{a, \gamma}$.   The Brauer--Siegel theorem states that, as $F$ runs through a sequence of number fields of given degree $n$ over $\Q$ and whose discriminants $\Delta_F$ tend, in absolute value, to $+\infty$, 
one has
\begin{equation} \label{eq.BS.cla}
 \forall \epsilon>0, \quad
 \Delta_F^{1/2-\epsilon} \ll_{n, \epsilon}
 |Cl(F)|\cdot R(F)
 \ll_{n, \epsilon}   \Delta_F^{1/2+\epsilon} \qquad(\text{as } |\Delta_F|\to\infty),
\end{equation}
where $Cl(F)$ denotes the class group of $F$ and $R(F)$ its regulator of units. 
At least formally, \eqref{eq.BS.alternative} is very similar to \eqref{eq.BS.cla}. A more detailed analogy is explained in \cite{Hindry_MW} and \cite{HP15}.
 
	\begin{proof}
	We know from Theorem \ref{theo.BSD} that the BSD conjecture holds for $E_{a, \gamma}$. 
	In particular, the Tate--Shafarevich group $\sha(E_{a, \gamma})$ is indeed finite and the special value $L^\ast(E_{a, \gamma}, 1)$ satisfies  \eqref{eq.BSD}. 
	The BSD formula \eqref{eq.BSD} and Proposition \ref{prop.dioph.bnd} then imply the estimate: 
	\begin{equation}\notag{}
	\frac{\log\big(|\sha(E_{a, \gamma})|\cdot \Reg(E_{a, \gamma})\big)}{\log H(E_{a, \gamma})} = 1 + \frac{\log L^\ast(E_{a, \gamma}, 1)}{\log H(E_{a, \gamma})} + o(1) 
	\qquad (\text{as }a\to\infty).
	\end{equation} 
	Therefore, to conclude the proof of Theorem \ref{theo.BS}, it remains to prove that $|\log L^\ast(E_{a, \gamma}, 1) | = o\big( \log H(E_{a, \gamma})\big)$ or, alternatively,  that 
	\[ |\log L^\ast(E_{a, \gamma}, 1) | = o\big( b(E_{a, \gamma})\big) 
	\qquad (\text{as }a\to\infty)\]
	because $\log H(E_{a, \gamma})$ and $b(E_{a, \gamma})$ have the same order of magnitude as $a\to\infty$ (see \S\ref{subsec.invariants}). 
	But we have already showed in Theorem \ref{theo.bnd.spval} that this asymptotic estimate holds.
	\ProofEnd\end{proof}


\noindent\hfill\rule{7cm}{0.5pt}\hfill\phantom{.}

\paragraph{Acknowledgements} 
 It is a pleasure to thank Marc Hindry and Douglas Ulmer for their encouragements and for their comments on earlier versions of this work.
 The author would also like to thank
 Bruno Anglès, 
 Peter Bruin and 
 Peter Koymans for fruitful discussions about various parts of the paper. 
This work has been carried out at Universiteit Leiden whose financial support and perfect working conditions are gratefully acknowledged.
The author is also partially supported by the ANR Grant ANR-17-CE40-0012 (FLAIR).

\newcommand{\mapolicebackref}[1]{\hspace*{-2pt}{\textcolor{orange}{\footnotesize $\uparrow$ #1}}}
\renewcommand*{\backref}[1]{\mapolicebackref{#1}}
\hypersetup{linkcolor=orange!80}

\small
\bibliographystyle{alpha}
\bibliography{../../Biblio_GENERAL.bib} 

\normalsize

\noindent\rule{7cm}{0.5pt}

\smallskip
 
{\sc Mathematisch Instituut, Universiteit Leiden,} PO Box 9512, 2300 RA Leiden, The Netherlands.
 
{\it E-mail:} \href{mailto:r.m.m.griffon@math.leidenuniv.nl}{{r.m.m.griffon@math.leidenuniv.nl}}

\end{document}